\documentclass[a4paper,10pt]{article}

\usepackage{amsmath,amssymb,amsthm,graphicx,
	subfig,%showkeys,
verbatim,%refcheck
}

\usepackage{geometry}
\usepackage{tikz}
\usetikzlibrary{shapes.geometric,calc,decorations.markings,math}

\tikzstyle{nodo}=[circle,draw,fill,inner sep=0pt,minimum size=%
1.5mm]
\tikzstyle{infinito}=[circle,inner sep=0pt,minimum size=0mm]

\newcommand\R{{\mathbb R}}

\newcommand\N{{\mathbb N}}
\newcommand\Hmu{{H_\mu^1}}
\newcommand\T{\mathcal T}
\newcommand\LL{\mathcal L}
\newcommand\f{\frac}
\newcommand\dx{{\,dx}}

\newcommand{\uLp}{\|u\|_p}
\newcommand{\uLtwo}{\|u\|_2}
\newcommand{\udot}{\|u'\|_2}

\newcommand\Lp{{\Lambda^+}}
\newcommand\Lm{{\Lambda^-}}

\newcommand\elevel{{\mathcal E}}
\newcommand\ee{{\mathcal E}}
\newcommand\II{{\mathcal I}}

\newcommand\G{\mathcal G}

\newcommand\K{\mathcal K}

\newcommand{\uell}{u_\ell}
\newcommand{\vell}{v_\ell}\newcommand{\All}{{A_\ell}}
\newcommand{\Gell}{{\G_\ell}}
\newcommand{\Gz}{{\G_0}}

\newcommand\vv{\textsc{v}}
\newcommand\ww{\textsc{w}}

\newcommand\er{\tt r}
\newcommand\es{\tt s}
\newcommand\ti{\tt t}
\newcommand\ed{\tt e}
\newcommand\hh{\tt h}

\newcommand{\Vuno}{{V_{\ti}}}
\newcommand{\Vdue}{{V_{\es}}}

\newcommand{\Funo}{{\mathcal F_{\ti}}}
\newcommand{\Fdue}{{\mathcal F_{\es}}}

\newcommand\eps{\varepsilon}

\newtheorem{theorem}{Theorem}[section]
\newtheorem{proposition}[theorem]{Proposition}
\newtheorem{lemma}[theorem]{Lemma}

\theoremstyle{remark}
\newtheorem{remark}[theorem]{Remark}
\newtheorem*{remark*}{Remark}

\theoremstyle{definition}
\newtheorem{definition}[theorem]{Definition}

\date{}

\title{Uniqueness and non--uniqueness of prescribed mass \\ NLS ground states on metric graphs}

\author{Simone Dovetta$^\sharp$, Enrico Serra$^\dagger$,
Paolo Tilli$^\dagger$ \\ \ \\
{\small$^\sharp$Istituto di Matematica Applicata e Tecnologie Informatiche "E. Magenes"} \\ {\small Consiglio Nazionale delle Ricerche,
via Adolfo Ferrata, 1, ​27100 Pavia, Italy}
\\ \ \\{\small$^\dagger$Dipartimento di Scienze
Matematiche ``G.L. Lagrange'', Politecnico di Torino } \\ {\small
Corso Duca degli Abruzzi, 24, 10129 Torino, Italy}}

\begin{document}

\maketitle

\begin{abstract} We consider the problem of uniqueness of ground states of prescribed mass for the Nonlinear Schr\"odinger Energy with power nonlinearity on noncompact metric graphs. We first establish that the Lagrange multiplier appearing in the NLS equation is constant on the set of ground states of mass $\mu$, up to an at most countable set of masses. Then we apply this result to obtain uniqueness of ground states on two specific noncompact graphs. Finally we construct a graph that admits at least two ground states with the same mass having different Lagrange multipliers. Our proofs are based on careful variational arguments and rearrangement techniques, and hold both for the subcritical range $p\in (2,6)$ and in the critical case $p=6$.

\end{abstract}

\smallskip


\section{Introduction}
\label{sec:intro}
	
This paper is devoted to the study of uniqueness of ground states of prescribed mass for the Nonlinear Schr\"odinger Equation on noncompact metric graphs. 

A metric graph with a finite number of edges is called noncompact if at least one edge is unbounded (i.e. it is a half-line). Given one such graph $\G$ we consider the NLS energy functional

\begin{equation}
\label{NLSe}
E (u,\G) =\frac 1 2 \int_\G |u'|^2\dx -\frac 1 p \int_\G |u|^p\dx
\end{equation}
together with the {\em mass constraint}
\begin{equation}
\label{mass}
\| u \|^2_{L^2 (\G)} \ = \ \mu
\end{equation}
and we focus on functions that satisfy the constraint \eqref{mass} and minimize the energy functional $E$.
Throughout the paper the exponent $p\in (2,6]$ is fixed,
while the mass $\mu$ is a parameter of the problem.

In this framework, by a ``ground state of mass $\mu$'' we mean a solution to
the minimization problem
\begin{equation*}
\label{introminprob}
\ee_\G(\mu) := \inf \left\{ E(u,\G)   \; : \; u \in H^1(\G),\; \Vert u\Vert_{L^2(\G)}^2=\mu\right\}
\end{equation*}
and we call $\ee_\G(\mu)$ the ground state level at mass $\mu$. 
Ground states solve the stationary NLS equation
\begin{equation}
\label{introeq}
u'' + |u|^{p-2}u = \lambda u
\end{equation}
on every edge of $\G$, and satisfy the Kirchhoff conditions at every  vertex (namely, the sum of the outgoing derivatives of $u$ at every vertex is zero). The number $\lambda = \lambda(u)$ appearing in \eqref{introeq}, which will play a very important role in this paper, is interpreted as a Lagrange multiplier, that arises due to the mass constraint \eqref{mass}.

Schr\"odinger equations on metric graphs have been extensively investigated through the years, both in the linear setting (see the seminal paper \cite{KS} and for instance \cite{BKKM, BLS,EFK,EP,KKMM} and references therein for recent developments) and in the nonlinear one (we refer to the two reviews \cite{ASTparma,N} for detailed discussions on the topic and a detailed bibliography, as well as the references below).

The existence of constrained critical points of \eqref{NLSe} on metric graphs is nowadays fairly well understood for a wide class of graphs (see \cite{BMP,dovjde,g,mp} for compact graphs, \cite{acfn1,acfn2,ast, ast2, ast3,CFN2,KMPX,NP,nps,PSV} for noncompact graphs with half--lines, \cite{ad, adr,adst,dovetta,pa,ps} for periodic graphs and \cite{dst} for metric trees, as well as \cite{dt,dt-p,st,st2,t} for the model involving concentrated nonlinearities). On the contrary, hardly anything is known on the uniqueness of ground states. Up to our knowledge, only few partial results are available in highly specific contexts (see \cite{cds,cfn,NP,nps}). However, in most  cases uniqueness arises as a by--product of some bifurcation analysis, and not as a systematic investigation of the problem.

The lack of uniqueness results is due to a number of reasons. Starting at a general level, it is well known that uniqueness  for nonlinear problems is very hard to obtain. First because uniqueness (modulo symmetries) might actually fail for large classes of problems (see for instance \cite{dancer} and references therein). Secondly, in our opinion, the lack of general techniques and the need to proceed with {\em ad hoc} arguments depending on the specific equation under consideration pose serious difficulties that can only be overcome by a deep understanding of the problem. Moving from the milestone paper \cite{gnn}, many different techniques have been developed in several contexts. Typical uniqueness arguments may be based on dimensional reduction (for instance by showing that solutions are ``one dimensional''), or on perturbative arguments, when a parameter is involved, or on nondegeneracy properties, or on scaling when the domain permits that, or on blow-up techniques and so on. At times, a simple device as that of subtracting equations and operating on the difference of solutions may yield fruitful information leading to uniqueness.

If one is interested in establishing uniqueness for NLS ground states of prescribed mass, as in this paper, then basically none of the preceding arguments works. The domain has in general no symmetries, the problem is not perturbative  and scaling is not permitted since it alters the lengths of the edges. A further, and fundamental, difficulty is caused by the presence of the mass constraint, which prohibits many of the manipulations on ground states that one could try. Even worse, two ground states with the same mass may solve {\em different} equations, since nothing guarantees that two such ground states should share the same Lagrange multiplier.

For these reasons, in general, we do not expect that ground states with fixed mass are unique. The problems outlined above and the topology of a generic graph, which may be extremely complicated, might well lead to the coexistence of ground states with the same mass, as these could be localized on faraway parts of the graph, and minimize the energy, so to speak, by completely different reasons.

In this paper we make a first attempt to bring order in these questions, by addressing a number of problems. We now describe our main results, postponing precise statements and definitions to the next Section.

In our first result we work in an interval $J$ of masses for which existence of ground states is guaranteed. Theorem \ref{thm1} establishes some relations between the set of ground states $u$ of mass $\mu$, called $M(\mu)$, the possible values of $\lambda = \lambda(u)$ on  this set  and the derivative of the ground state energy level  $\ee_\G(\mu)$. More precisely,  it turns out that $\lambda(u)$ is constant on $M(\mu)$, except possibly for a countable set $Z$ of masses; that is to say, ground states with the same mass (outside $Z$) need not be unique, but they all share the same Lagrange multiplier. This allows us to define a function $\overline \lambda : J\setminus Z \to \R$ by setting $\overline \lambda(\mu)$ equal to $\lambda(u)$, for any $u\in M(\mu)$. We prove, among other things, that  $\overline \lambda$ is strictly increasing and that $\ee_\G'(\mu) =-\overline \lambda(\mu)/2$. This result serves as the starting point for the rest of the paper.
We would like to point out that this first result is ``abstract'' in nature, since it does not use specifically any property of graphs. For this reason, it can be adapted to other context such as, for example, that of  NLS ground states on subsets of $\R^N$.

Next we consider two specific examples of graphs (see Figure \ref{fig-NlT}). For each of these we prove, in Theorems \ref{thm2} and \ref{thm3}, that outside the countable set $Z$, ground states of mass $\mu$ are indeed unique. To our knowledge these are the first general uniqueness results for NLS ground states on graphs. Furthermore, we highlight the fact that Theorem \ref{thm3} establishes uniqueness of ground states of prescribed mass at the $L^2$--critical power $p=6$ and for a whole interval of masses. This phenomenon appears to be remarkably new and in sharp contrast with the usual features of the critical regime on the real line (see Remark \ref{remsol}).

The final result addresses the question of the existence of the set $Z$, and its proof is the most involved part of the paper. In  Theorem \ref{thm4} we construct a graph $\G$ (Figure \ref{grafone}), for which the set $Z$ is not empty, thus showing that Theorem \ref{thm1} cannot be substantially improved in general. In particular, we prove that for every $\mu$ there exist suitable values of the lengths of the edges of  $\G$ for which two ground states of mass $\mu$ exist on $\G$ and furthermore, these two ground states have different Lagrange multipliers.

The proofs of the main theorems use several results established in preceding papers. Whenever possible we have tried to make the proofs self-consistent. In some other cases we refer the reader to the papers where the needed results were proved.
\medskip

The structure of the paper is as follows. In Section \ref{state} we introduce notation, the precise setting, the statement of the main results and we review some of the fundamental tools that are used throughout the paper. Section \ref{compactsec} collects the main compactness properties of minimizing sequences that are used in the rest of the paper.
Section \ref{sec:lambda} is devoted to the proof of Theorem \ref{thm1}, while Section \ref{uni} contains the proofs of  the main uniqueness results, Theorems \ref{thm2} and \ref{thm3}. The non-uniqueness result, Theorem \ref{thm4}, is proved in Section \ref{nonuni}. Finally, an Appendix contains some technical results that are needed in the preceding Sections.

\section{Notation, setting and main results}
\label{state}

Throughout this paper, by noncompact metric graph we mean a connected metric graph containing a finite number of edges, at least one of which is unbounded. 
As is always the case in the literature, bounded edges are identified with intervals of the real line, while unbounded edges are identified with (copies of) the interval $[0,+\infty)$. 

The function spaces $L^p(\G)$, $H^1(\G)$ and so on are defined in the usual way (see e.g. Section 2 of \cite{ast} and the monograph \cite{BK}). 
When the underlying graph is clearly understood, we use symbols like $\uLp,\udot, \dots$ to denote $\|u\|_{L^p(\G)}$, $\|u'\|_{L^p(\G)}, \dots$, and likewise for similar expressions involving the dependence on $\G$. Whenever necessary, the graph is instead written explicitly to avoid ambiguity.

Given a noncompact metric graph $\G$ and a number $p\in (2,6]$, we consider the NLS energy functional $E(\,\cdot\,,\G) : H^1(\G) \to \R$ defined by
\begin{equation}
\label{energy}
E(u,\G):=\frac12 \|u'\|_2^2-\frac1p \|u\|_p^p=\frac12 \int_\G|u'|^2\dx-\frac1p \int_\G|u|^p\dx.
\end{equation}
For any $\mu>0$ we introduce the set of {\em mass} $\mu$ functions in $H^1(\G)$, defined by
\[
\Hmu(\G):=\{\,u\in H^1(\G)\,:\,\|u\|_2^2=\mu\,\}
\]
and we set
\begin{equation}
\label{elevel}
\elevel_\G(\mu):=\inf_{u\in\Hmu(\G)}E(u,\G).
\end{equation}
By ``ground state'' of mass $\mu$ we mean any function $u\in H_\mu^1(\G)$ satisfying
\begin{equation}
\label{minprob}
E(u,\G) = \elevel_\G(\mu).
\end{equation}
In studying the minimization problem \eqref{minprob} it is clearly sufficient to work with nonnegative functions, which we will do throughout the paper, without further explicit reference. In particular, when we talk about {\em uniqueness} of ground states, we always mean that the {\em positive} solution to \eqref{minprob} is unique. 

Obviously ground states solve, for some Lagrange multiplier $\lambda \in \R$, the stationary NLS equation
\begin{equation}
\label{equation}
u''+|u|^{p-2}u=\lambda u
\end{equation}
on each edge of $\G$, with Kirchhoff boundary conditions at the vertices (see e.g. \cite{ast}).	

\begin{remark}[Solitons]
\label{remsol} 
When $\G=\R$ the (positive) solutions to  \eqref{minprob}
are called {\em solitons} and, in the $L^2$--subcritical regime $p\in(2,6)$, they  
are unique up to translations. We denote by  $\phi_\mu$
the soliton of mass $\mu$ centered at the origin. It obeys the scaling rule 
\[
\phi_\mu(x)=
\mu^\alpha \phi_1\bigl ( \mu^\beta x\bigr ),\quad
\alpha=\frac 2{6-p},\quad
\beta=\frac {p-2} {6-p},
\]
where $\phi_1(x)=c \mathop{\rm sech}(C x)^{\alpha/\beta}$ with $c, C>0$ depending only on $p$
(note that $\alpha,\beta>0$ when $p\in (2,6)$). Then a direct computation shows that
\[
\elevel_\R(\mu)=
E(\phi_\mu,\R)=-\theta_p \mu^{2\beta+1},\quad
\theta_p:=-E(\phi_1,\R)>0.
\]
When $\G=\R^+$, the unique positive ground state of mass $\mu$ is the ``half soliton'', i.e. $\phi_{2\mu}$
restricted to $\R^+$, so that
\begin{equation}
\label{energiamezzosol}
\elevel_{\R^+}(\mu)=E(\phi_{2\mu},\R^+)=\frac 1 2E(\phi_{2\mu},\R) =-2^{2\beta}\theta_p \mu^{2\beta+1}
\end{equation}
with $\theta_p$ as above. Note that in each case, the ground state level is always negative. 

At the $L^2$--critical exponent $p=6$, when $\G=\R$ there exists a critical mass $\mu_\R=\f{\sqrt{3}\pi}{2}$ such that
\begin{equation}
\label{livelliretta}
\elevel_\R(\mu)=\begin{cases}
0 & \text{if }\mu\leq\mu_\R\\
-\infty & \text{if }\mu>\mu_\R
\end{cases}
\end{equation}
and a family of solitons, invariant under mass-preserving dilations, exists if and only if $\mu=\mu_\R$. When $\G=\R^+$ the behavior is analogous, the threshold value of the mass  being $\mu_{\R^+}=\mu_\R/2$.
\end{remark}

The issue of existence of ground states of prescribed mass for noncompact graphs $\G$ has been addressed in \cite{ast,ast2} for $p\in(2,6)$ and in \cite{ast3} for $p=6$. In the former case, it has been shown (Theorem 3.3 in \cite{ast2}) that the condition
\begin{equation}
\label{suff}
\elevel_\G(\mu) < \elevel_\R(\mu)
\end{equation}
is always sufficient for the existence of ground states. This condition is {\em almost} necessary, since the weak inequality always holds (Theorem 2.2 in \cite{ast}).

When $p=6$, the existence of ground states occurs only for certain classes on graphs  that we now recall and for specific values of the mass. We refer to \cite{ast3} for the terminology appearing in the following definition.

\begin{definition}
\label{AB} 
Let $\G$ be a noncompact graph. We say that $\G$ is 
\begin{itemize}
\item[$i)$] of type $A$, if $\G$ has exactly one half-line and no terminal edge;
\item[$ii)$] of type $B$, if $\G$  has at least two half-lines, no terminal edge and cannot be covered by cycles.
\end{itemize}
\end{definition}

In \cite{ast3} it has been proved that for both types of graph there exists a number $\mu_\G \in [\mu_\R^+, \mu_\R]$, called the {\em critical mass} of the graph, that determines the existence of ground states in the following way.

\begin{theorem}[\cite{ast3}] 
\label{ex6}
Let $p=6$ and let $\G$ be a noncompact graph. Then
\begin{itemize}
\item [i)] if $\G$ is of type $A$, then $\mu_\G = \mu_{\R^+}$ and ground states of mass $\mu$ exist if and only if $\mu \in (\mu_\G,\mu_\R]$;
\item[ii)]  if $\G$ is of type $B$ and $\mu_\G < \mu_\R$, then ground states of mass $\mu$ exist if and only if $\mu \in [\mu_\G,\mu_\R]$,
\end{itemize}
Furthermore,
\begin{equation}
\label{livelli6}
\ee_\G(\mu) \begin{cases} = 0 & \text{ if } \mu \le \mu_\G \\
< 0 & \text{ if } \mu \in (\mu_\G,\mu_\R] \\
=-\infty  & \text{ if } \mu > \mu_\R.
\end{cases}
\end{equation}
\end{theorem}

\begin{remark} 
\label{condlevel}
$i)$ In the preceding Theorem we see that in the critical case $p=6$ one can reverse the point of view and state a sufficient condition for the existence of ground states in terms of the energy instead of the mass. Precisely, as we see from its statement and \eqref{livelliretta}, the condition $\ee_\G(\mu) < \ee_\R(\mu) = 0$ is always sufficient for the existence of ground states. It is however not necessary for graphs of type $B$ and $\mu=\mu_\G$.

$ii)$ It has been proved in \cite{ast2} and \cite{ast3} that $\ee_\G$, as a function of $\mu$, is continuous and concave on $\R^+$ if $p\in (2,6)$ and on $(\mu_\G,\mu_\R]$ if $p=6$. 	
\end{remark}

Two main tools in the study of \eqref{minprob} are ubiquitous. The first is given by the  Gagliardo--Nirenberg inequalities
\begin{equation}
\label{GN}
\uLp^p\leq K\uLtwo^{\f p2+1}\udot^{\f p2-1}, \qquad K=K(\G,p),
\end{equation}
that hold for every $u\in H^1(\G)$ and every $p\geq2$, and their $L^\infty$ version
\[
\|u\|_{\infty}^2\leq K\uLtwo \udot , \qquad K=K(\G).
\]
For a detailed proof of these inequalities, see \cite{ast2}. When $p=6$ these inequalities are no longer sufficient to ensure boundedness of minimizing sequences. Instead, it is convenient to use the following stronger form, proved in Lemma 4.4 in \cite{ast3}: if $\G$ is a noncompact graph with no terminal edge and $\mu \in (0,\mu_\R]$, then for every $u\in H_\mu^1(\G)$ there exists a number $\theta=\theta(u)\in[0,\mu]$ such that
\begin{equation}
	\label{modGN}
	\|u\|_6^6\leq3\left(\f{\mu-\theta}{\mu_\R}\right)^2\|u'\|_2^2+C\theta^{1/2}, \qquad C=C(\G).
\end{equation}
The other tool is provided by the body of rearrangement techniques on graphs, for which we refer to Section 3 of \cite{ast}. Rearrangements of $H^1$ functions on graphs are very frequently used in this paper and the reader is expected to be familiar with them. For the reader's convenience, we recall here the definitions of \emph{decreasing} and \emph{symmetric} rearrangement of a function $u\in H^1(\G)$. Letting for simplicity $u\geq0$, $u\neq0$ on $\G$ and introducing the distribution function
\[
\rho(t):=\left|\{x\in\G\,:\,u(x)>t\}\right|\,,\qquad t\geq0,
\] 
then
\begin{itemize}
	\item[(i)] the decreasing rearrangement $u^*:\left[0,|\G|\right]\to\R$ is 
	\[
	u^*(x):=\inf\{t\geq0\,:\,\rho(t)\leq x\}\,;
	\]
	\item[(ii)] the symmetric rearrangement $\widehat{u}:\left[-|\G|/2,|\G|/2\right]$ is 
	\[
	\widehat{u}(x):=\inf\{t\geq0\,:\,\rho(t)\leq2|x|\}\,,
	\]
\end{itemize}
where $|\G|$ denotes the total length of $\G$. The functions $u^*$ and $\widehat u$ are equimeasurable with $u$, so that all $L^q$ norms are preserved when $u$ is rearranged, and the classical Polya--Szeg\H{o} inequality $\|(u^*)'\|_2 \le \|u'\|_2$ holds. Moreover, if a nonnegative $u\in H^1_\mu(\G)$ has at least $2$ preimages for almost every value in $(0,\|u\|_\infty)$, i.e. if
\[
\#\{x \in \G \; : \; u(x) = \tau \} \ge 2\qquad \text{for almost every } \tau \in (0,\|u\|_\infty),
\]
then $\|(\widehat u)'\|_2 \le \|u'\|_2$.

We also quote, for further reference, Lemma 2.1 of \cite{astbound}: if $p\in(2,6)$ and if a nonnegative $u\in H^1_\mu(\G)$ has at least $N\ge 2$ preimages for almost every value in $(0,\|u\|_\infty)$, then
\begin{equation}
\label{Ncontr}
E(u,\G) \ge - \theta_p \left(\frac2N \right)^{2\beta} \mu^{2\beta+1} = \left(\frac2N \right)^{2\beta}\ee_\R(\mu).
\end{equation}
This inequality still holds when $p=6$ but for $\mu \le \mu_\R$ takes the weaker form $E(u,\G) \ge 0$, since $\ee_\R(\mu) = 0$ in this case.

To state our first result, we introduce the set of ground states with mass $\mu$
\begin{equation}
\label{emme}
M(\mu) = \left\{u \in \Hmu(\G) \; :\;  E(u,\G) = \ee_\G(\mu)\right\}
\end{equation}
and we define the functional  $\LL:H^1(\G)\setminus\{0\} \to\R$  by
\begin{equation}
\label{lambda u}
\LL(u)=\frac{\uLp^p-\udot^2}{\uLtwo^2}.
\end{equation}
Note  that $\LL$ is continuous and that if $u\in H^1(\G)$ is a (nonzero) solution of \eqref{equation} for a certain $\lambda$, then $\LL(u)=\lambda$, as one immediately sees by multiplying \eqref{equation} by $u$ and integrating over $\G$.

Finally, for every $\mu$ such that $M(\mu)\neq\emptyset$ we set
\begin{equation}
\label{Lpm}
\Lm(\mu) =\inf_{u\in M(\mu)} \LL(u), \qquad \Lp(\mu) =\sup_{u\in M(\mu)} \LL(u)\,. 
\end{equation}
Our first result describes the relationships between the ground state energy function $\ee_\G(\mu)$, the set $M(\mu)$ and the Lagrange multipliers associated with elements of $M(\mu)$.

\begin{theorem}
\label{thm1} 
Let $\G$ be a noncompact graph and assume that $J\subset \R^+$ is an interval of masses such that 
\[
\ee_\G(\mu) < \ee_\R(\mu).
\]
($J$ is allowed to  contain $\mu_\G$ if $p=6$ and $\G$ is of type $B$). Then there exists an at most countable set $Z\subset J$ such that, setting
\[
I = J \setminus Z,
\]
\begin{itemize}
\item[i)] for every $\mu \in I$
\[
\Lm(\mu) = \Lp(\mu)\,,
\]
namely $\LL$ is constant on $M(\mu)$; 
\item[ii)] the function $\overline\lambda :I \to \R$ defined by $\overline\lambda (\mu) = \LL(u)$, $u\in M(\mu)$, is strictly increasing;

\item[iii)] for every $\mu \in J$,
\begin{equation}
\label{derlevel}
(\elevel_\G)_+'(\mu) = -\frac12 \Lp(\mu),\qquad \ (\elevel_\G)_-'(\mu) = -\frac12 \Lm(\mu).
\end{equation}
In particular, $\ee_\G$ is differentiable at $\mu$ if and only if  $\mu \in I$, in which case
\[
\elevel_\G'(\mu) = -\frac12 \overline \lambda (\mu).
\]
\end{itemize}
\end{theorem}

\begin{remark} Concretely, in view of \eqref{suff} and Theorem \ref{ex6}, the interval $J$ of the preceding result may be 
$a)$ any interval of masses where $\ee_\G(\mu) < \ee_\R(\mu)$, if $p\in(2,6)$; $b)$ the interval $(\mu_\G,\mu_\R]$, if $\G$ is of type $A$; the interval $[\mu_\G,\mu_\R]$, if $\G$ is of type $B$.
\end{remark}

\begin{figure}[t]
\centering
\subfloat[][]{
\begin{tikzpicture}[xscale= 0.5,yscale=0.5]
\node at (0,0) [nodo] (00) {};
\node at (0,2) [nodo] (02) {};
\node at (-4,-5) [infinito] (inf1) {};
\node at (-1.5,-6) [infinito] (inf2) {};
\node at (1.5,-6) [infinito] (inf3) {};
\node at (4,-5) [infinito] (inf4) {};
\node at (0.4,1.1) [infinito] {$\ti$};
\node at (-0.5,0.2) [infinito] {$\vv$};
\draw[-] (00)--(02);
\draw[-] (00)--(inf1);
\draw[-] (00)--(inf2);
\draw[-] (00)--(inf3);
\draw[-] (00)--(inf4);
\draw[-,dashed] (inf1)--(-4.7,-5.875);
\draw[-,dashed] (inf2)--(-1.8,-7.2);
\draw[-,dashed] (inf3)--(1.8,-7.2);
\draw[-,dashed] (inf4)--(4.7,-5.875);
\end{tikzpicture}}\qquad
\subfloat[][]{
\begin{tikzpicture}[xscale= 0.5,yscale=0.5]
\node at (0,0) [nodo] (00) {};
\node at (-7,0) [infinito] (-55) {};
\draw[-] (00)--(-55);
\draw[-,dashed] (-55)--(-7.85,0);
\draw (2,0) circle (2); 
\node at (-0.3,0.4) [infinito] {$\vv$};
\node at (3.4,2) [infinito] {$\es$};
\end{tikzpicture}}
\caption{A graph with $N=4$ half--lines and a terminal edge $\texttt{t}$ of length $t$ (a); a tadpole graph (b).}
\label{fig-NlT}
\end{figure}
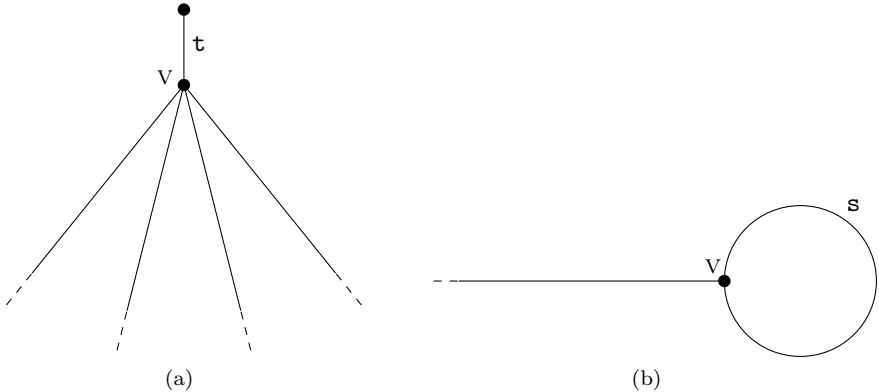

Some comments are in order. First, we point out that, as one can see from its proof, the preceding result is in fact quite general. Indeed, even though it is stated in the context of NLS ground states on metric graphs, it could be easily transferred to other settings, such as, for instance, that of NLS ground states on subsets of $\R^N$. In that case it could be considered as the starting point towards the study of uniqueness of ground states under a mass constraint, a topic for which general results have not yet been obtained. 

Secondly, Theorem \ref{thm1} raises at least two natural questions. To begin with, the issue of uniqueness of ground states when $\mu \in I$. In this case, all ground states of mass $\mu$ have the same energy and the same Lagrange multiplier. This seems to suggest that the ground state is in fact unique but, due to the lack of general techniques,  we believe that at this point the study should be carried out on a case by case basis, working each time on a specific graph. Furthermore, we suspect that there could still be  ``exceptional'' situations in which uniqueness might  fail. The next two results deal with this problem for two classes of graphs, and answer in the affermative the question of uniqueness for $\mu \in I$. A second natural question concerns  the set $Z$. In the second part of the paper (starting in Section \ref{nonuni}) we will produce an example of graph for which $Z$ is not empty, by constructing two ground states having the same mass but different Lagrange multipliers. This shows that, without further assumptions, the result of Theorem \ref{thm1} is somewhat sharp, in the sense that it is impossible in general to rule out the presence of ground states with the same mass and different Lagrange multipliers.

In the next two theorems, the existence results have been reported for the sake of completeness, but they are in fact well known (see \cite{ast,ast2,ast3} and the proofs in Section \ref{uni} for more details).

\begin{theorem}
\label{thm2}  
Let $p \in(2,6)$ and let $\G_{N,t}$ denote the graph made up of $N\geq2$ half-lines
and a terminal edge $\ti$ of length $t$, all emanating from the same vertex $\vv$
(Figure \ref{fig-NlT}.a). 
For every fixed $t>0$ there exists $\overline\mu=\overline\mu(t)$ such that ground states of mass $\mu$ on $\G_{N,t}$ exist if and only if $\mu\geq\overline\mu$. The function $\mu\mapsto\elevel_{G_{N,t}}(\mu)$ is strictly decreasing on $[\overline\mu,+\infty)$ and
\begin{equation}
\label{livelli_mu}
\elevel_{G_{N,t}}(\mu)\begin{cases}
=\elevel_\R(\mu) & \text{if }\mu\leq\overline\mu\\
<\elevel_\R(\mu) & \text{if }\mu>\overline\mu\,.
\end{cases}
\end{equation} 
Furthermore, 
\begin{itemize}
\item[(i)] if $N=2$, then $\overline\mu=0$;
\item[(ii)] if $N\geq3$, then $\overline\mu>0$.
\end{itemize}
In both cases, for all but at most countably many values of $\mu \in [\overline\mu,+\infty)$, the ground state of mass $\mu$  is unique.
\end{theorem}
With the notation introduced in Theorem \ref{thm1}, this result states that for every fixed $t>0$ there results $J = [\overline\mu,+\infty)$ and the ground state of mass $\mu$ is unique  for every $\mu \in I$.

\begin{theorem}
\label{thm3}
Let $p \in(2,6]$ and let $\G$ be a tadpole graph (Figure \ref{fig-NlT}.b). Then ground states of mass $\mu$ exist for every $\mu>0$ if $p\in(2,6)$ and for every $\mu\in(\mu_{\R^+},\mu_\R]$ if $p=6$. For all but at most countably many values of $\mu$ in these intervals, the ground state of mass $\mu$  is unique.
\end{theorem}
In this case, therefore, we have $J=\R^+$ when $p\in(2,6)$, while $J=(\mu_{\R^+},\mu_\R]$ when $p=6$ and, again,  for every $\mu \in I$, the ground state of mass $\mu$ is unique.

\begin{remark}
The recent paper \cite{NP} provides numerical simulations (see Figure 1 in that paper) suggesting uniqueness of ground states on the tadpole graph at the critical power $p=6$ and in the interval of masses $(\mu_{\R^+},\mu_\R]$ where existence is granted. Theorem \ref{thm3} above gives a rigorous proof of this uniqueness result.
\end{remark}

\begin{figure}[t]
\begin{center}
\begin{tikzpicture}[xscale= 0.5,yscale=0.5]
\node at (0,0) [nodo] (00) {};
\node at (7,2) [nodo] (02) {};
\node at (-2.5,2.5) [nodo] (B2) {};
\node at (-4.4,-5) [infinito] (inf1) {};
\node at (-1.5,-6) [infinito] (inf2) {};
\node at (1.5,-6) [infinito] (inf3) {};
\node at (4,-5) [infinito] (inf4) {};
\node at (3.8,1.6) [infinito] {$\tt r$};
\node at (-1.5,0.8) [infinito] {$\tt t$};
\node at (9.5,.9) [infinito] {$\tt s$};
\node at (0.1,.6) [infinito] {$\vv$};
\node at (6.8,1.4) [infinito] {$\ww$};
\draw[-] (00)--(02);
\draw[-] (00)--(inf1);
\draw[-] (00)--(inf2);
\draw[-] (00)--(inf3);
\draw[-] (00)--(inf4);
\draw[-,dashed] (inf1)--(-5.16,-5.875);
\draw[-,dashed] (inf2)--(-1.8,-7.2);
\draw[-,dashed] (inf3)--(1.8,-7.2);
\draw[-,dashed] (inf4)--(4.7,-5.875);
\draw[-] (00)--(B2);
%\draw (-3.5,3.5) circle (1.414);
\draw (8.4,2.5) circle (1.5);
\end{tikzpicture}
\end{center}
\caption{The graph of Theorem \ref{thm4}: $N$ half--lines, a terminal edge $\ti$, a bounded edge $\er$ and a self-loop $\es$ attached to $\er$.}
\label{grafone}
\end{figure}
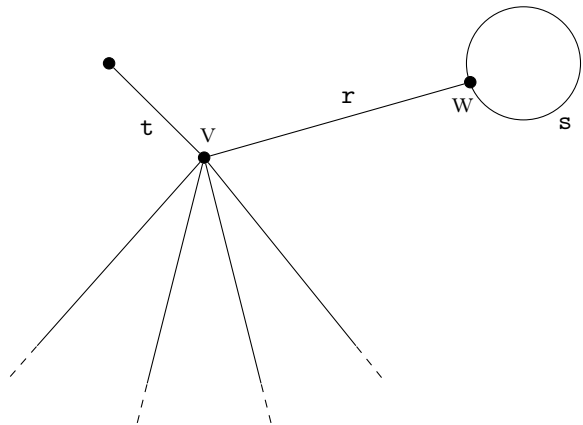

Our last result complements the two preceding Theorems. We construct a graph for which the set $Z$ is not empty, showing that for some value of $\mu$ there exist two ground states of mass $\mu$ having different Lagrange multipliers.

Precisely, we consider the graph $\G$ of Figure \ref{grafone}. It is made up of $N\ge 2$ half-lines, a terminal edge $\ti$ and a bounded edge $\er$ all emanating from a single vertex $\vv$. A self-loop $\es$ is attached at the tip $\ww$ of the bounded edge $\er$. We denote by $r,s,t$ the lengths of the edges $\er, \es, \ti$ respectively.

\begin{theorem}
\label{thm4}
Let $p\in(2,6)$. For every $\mu>0$ there exist positive numbers $r,s,t$ such that  the graph $\G$   of Figure \ref{grafone} with edges $\er, \es, \ti$ of lengths $r,s,t$ admits two ground states $u, v \in H_\mu^1(\G)$ that solve \eqref{equation} with different values of $\lambda$.
\end{theorem}

\section{Compactness  properties}
\label{compactsec}

The analysis of the compactness properties of minimizing sequences, or of sequences of ground states, plays a central role in many of our results. Many compactness Theorems are available but they are somehow scattered in the literature, and proved under different assumptions. In this Section we synthetise these results, to make the exposition as self-contained as possible.
We start by recalling the boundedness properties of sequences $(u_n)\subset H_\mu^1(\G)$ with bounded energy. We report the proof for completeness.

\begin{lemma}
\label{bdedseq} 
Let  either $p\in(2,6)$ and $\mu >0$ or $p=6$ and $\mu\in (0,\mu_\R]$.  Let $(u_n)\subset \Hmu(\G)$ be a sequence such that for every $n$
\[
E(u_n,\G) \le C 
\]
and assume further that $C<0$ if $p=6$ and $\mu = \mu_\R$.
Then $u_n$ is uniformly bounded in $H^1(\G)$.
\end{lemma}

\begin{proof} Since $\|u_n\|_2^2 = \mu$ for every $n$, it is sufficient to show the uniform boundedness of  $\|u_n'\|_2$.
If $p\in (2,6)$, plugging \eqref{GN} in \eqref{energy} we obtain, for every $n$,
\[
C \ge E(u_n,\G) \ge \frac12 \|u_n'\|_2^2 -\frac{K}p \mu^{\frac{p+2}4}\|u_n'\|_2^{\frac{p-2}2}.
\]
Since $(p-2)/2<2$, this shows that $u_n'$ is uniformly bounded in $L^2(\G)$.

If $p=6$, using \eqref{modGN} instead of \eqref{GN}, and writing $\theta_n$ for $\theta(u_n)$, there results
\[
C \geq E(u_n,\G)\geq \frac12\|u_n'\|_2^2\left(1-\left(\frac{\mu -\theta_n}{\mu_\R}\right)^2\right)-\frac C6\theta_n^{1/2}.
\]
As $\theta_n \le \mu$ for all $n$, we see that if $\mu <\mu_\R$, the coefficient of $\|u_n'\|_2^2$ is bounded away from zero and therefore $\|u_n'\|_2$ is uniformly bounded. If $\mu = \mu_\R$ (in which case we assume that $C<0$), expanding the coefficient of $\|u_n'\|_2^2$, we have
\[
0> C \geq E(u_n,\G) \ge \frac12 \frac{\theta_n}{\mu_\R}\|u_n'\|_2^2\left(2-\f{\theta_n}{\mu_\R}\right)-\f C6\theta_n^{1/2}
\]
which shows that $\theta_n$ is bounded away from zero, implying that $\|u_n'\|_2$ is again uniformly bounded.

\end{proof}

The most general result describing the behavior of bounded sequences in $H^1_\mu(\G)$ is probably Theorem 2.4 in \cite{astbound}. It is formulated for the subcritical case $p<6$, but it is immediate to realize that it holds also for $p=6$. Since we use it several times in this paper we report it here stated for every $p\in (2,6]$.

\begin{theorem}[\cite{ast2}, Theorem 2.4]
\label{propgen}
Let either $p\in(2,6)$ and $\mu >0$ or $p=6$ and $\mu\in (0,\mu_\R]$ and let $(u_n)\subset \Hmu(\G)$ be a sequence such that $u_n\rightharpoonup u$ in $H^1(\G)$ and a.e. on $\G$. Set 
\begin{align}
\label{defm}
& m:=\mu-\Vert u\Vert_2^2\in [0,\mu], \\
& \II:=\liminf_n E(u_n,\G) \nonumber.
\end{align}
Then one of the following  alternatives occurs, depending on the value of $m$:
\begin{itemize}
 \item[(i)] $m=0$. Then $u_n\to u$ strongly in $H^1(\G)$  and $\;E(u,\G)\leq \II$.
\item[(ii)] $0<m<\mu$. Then
\begin{equation}
\label{stimaL}
\II>\min \left\{  \frac{\mu}m \ee_\R(m),\, E\left(\sqrt{\frac \mu{\mu-m}}u,\,\G\right)\right\}.
\end{equation}
\item[(iii)]  $m=\mu$. Then $u\equiv 0$ and $\II\geq  \ee_\R(\mu)$.
\end{itemize}
\end{theorem}

The next result describes the behavior of sequences of ground states with converging masses.

\begin{proposition}
\label{newcomp} Let either  
\begin{itemize}
\item[i)]  $p\in(2,6]$, $\mu >0$ and $\elevel_\G(\mu)<\elevel_\R(\mu)$, or
\item[ii)] $p=6$, $\G$  a graph of type $B$ and $\mu = \mu_\G$.
\end{itemize}
For every $n\in \N$, let  $u_n \in M(\mu_n)$, with $\mu_n \to \mu$ as $n\to \infty$.
Then, up to subsequences, $u_n \to u$  in $H^1(\G)$ for some $u \in M(\mu)$. In particular, $M(\mu)$ is compact.
\end{proposition}

Note that by the discussion following \eqref{suff} and Theorem \ref{ex6}, assumptions $i)$ and $ii)$ cover all the cases in which ground states are known to exist. Particularly, when $p=6$, assumption $i)$ implies $\mu \in (\mu_\G, \mu_\R]$.

\begin{proof} We split the proof according to whether $i)$ or $ii)$ holds.
\medskip

\noindent{\em i)} Define $v_n\in H^1_{\mu}(\G)$ by
\begin{equation}
\label{defvnun}
v_n:=\sqrt{\frac{\mu}{\mu_n}}\,\,u_n.
\end{equation}
Since the sequence $(u_n)$ is bounded in $H^1(\G)$ by Lemma \ref{bdedseq},
and since $\mu_n \to \mu$, as $n\to\infty$  we have
\begin{equation}
\label{vn_min}
\begin{aligned}
E(v_n,\G) &=\frac12\frac{\mu}{\mu_n}\int_\G|u_n'|^2\dx-
\frac1p\Big(\frac{\mu}{\mu_n}\Big)^{\frac{p}2}\int_\G|u_n|^p\dx
= E(u_n,\G) + o(1)\\ &=\elevel_\G(\mu_n)+o(1) = \elevel_\G( \mu)+o(1), 
\end{aligned}
\end{equation}
having used the continuity of $\elevel_\G$ (Remark \ref{condlevel}).
Thus $(v_n)$ is a bounded minimizing sequence for $E$ in $H^1_{\mu}(\G)$. As such, it admits a subsequence (not relabeled)
such that $v_n \rightharpoonup v$, for some $v \in H^1(\G)$. In view of Theorem \ref{propgen} we have to exclude that
$m:= \mu -\|v\|_2^2 \in (0,\mu]$. In our case, $\II =  \ee_\G(\mu)< \ee_\R(\mu)$, and we immediately see that it cannot be $m=\mu$, since then we would have $\II \ge \ee_\R(\mu)$. The case $m\in (0,\mu)$ is also impossible since both
$\II  > \frac{\mu}m \ee_\R(m) > \ee_\R(\mu)$ and $\II  > E\left(\sqrt{\frac \mu{\mu-m}}u,\,\G\right) \ge \ee_\G(\mu)$ are false.
Thus it must be $m=0$, and then $v_n \to v$ strongly in $H^1(\G)$ and $E(v,\G) = \ee_\G(\mu)$. By \eqref{defvnun}, the same happens for $u_n$.
\medskip

\noindent{\em ii)} We first define $v_n$ as in \eqref{defvnun} so that, by \eqref{vn_min} and Lemma \ref{bdedseq},  $v_n$ is a bounded minimizing sequence for $E$ in $H_{\mu_\G}^1(\G)$.
It has been proved in \cite{ast3} that not all minimizing sequences are precompact: indeed since $\ee_\G(\mu_G)=0$, {\em every} sequence $(u_n)\in H_\mu^1(\G)$ such that $\|u_n'\|_2\to 0$ is a minimizing sequence. Compactness is however recovered in \cite{ast3} for those minimizing sequences that satisfy
\begin{equation}
\label{un_to_3}
\frac{\|v_n\|_6^6}{\|v_n'\|_2^2}\to3\qquad\text{as }n\to+\infty
\end{equation}
(these sequences are optimizing sequences for the Gagliardo-Nirenberg inequality \eqref{GN}, see \cite{ast3}).

We now check that \eqref{un_to_3} holds. Once this is proved, the same argument of Theorem 3.4 in \cite{ast3} shows that, up to subsequences, $v_n$ converges strongly in $H^1(\G)$ to 
some $v\in M(\mu_\G)$, and the same happens then for $u_n$.

To prove \eqref{un_to_3}, assume that for some $\delta>0$ and  some subsequence (not relabeled), 
\[
\frac{\|v_n\|_6^6}{\|v_n'\|_2^2}\to 3-2\delta \qquad\text{as }n\to+\infty.
\]
Then by  \eqref{defvnun}, also
\[
\frac{\|u_n\|_6^6}{\|u_n'\|_2^2} = \frac{\mu_n^2}{\mu_\G^2}\frac{\|v_n\|_6^6}{\|v_n'\|_2^2} \to 3-2\delta,
\]
since $\mu_n \to \mu_\G$. Thus, for every $n$ sufficiently large, $\|u_n\|_6^6\le (3-\delta) \|u_n'\|_2^2$, from which it follows
\[
E(u_n,\G) \ge \frac12\|u_n'\|_2^2  -\frac{3-\delta}6 \|u_n'\|_2^2 = \frac{\delta}6 \|u_n'\|_2^2 >0
\]
which is impossible since, as $u_n\in M(\mu_n)$, $E(u_n,\G) \le 0$. We have thus checked that $(v_n)$ is an optimizing sequence for inequality \eqref{GN} and we conclude as described above. 
\end{proof}

Finally we prove a simple sufficient condition for the existence of minimizers of doubly constrained problems
that we will use in Section \ref{nonuni}. 

\begin{theorem}
\label{compact} Let $p \in(2,6)$, let $\G$ be a noncompact graph and let $\ed$ be a fixed bounded edge of $\G$. Set
\[
V= \{ u \in H^1_\mu(\G) \; : \; \|u\|_{L^\infty(\G)} = \|u\|_{L^\infty(\ed)}\}
\]
and 
\[
\II = \inf_{u\in V} E(u,\G).
\]
If $\, \II \le \ee_\R(\mu)$, then $\II$ is achieved.
\end{theorem}

\begin{proof} It is a simple application of Theorem \ref{propgen}.
Let $(u_n) \subset V$ be a minimizing sequence for $E$ on $V$. By Lemma \ref{bdedseq}, $(u_n)$ is bounded in $H^1(\G)$ and we can assume (passing to a subsequence) that $u_n$ converges to some $u$ weakly in $H^1(\G)$ and strongly in $L^\infty_{loc}(\G)$. Since $V$ is weakly closed, $u\in V$. Let $m = \mu - \|u\|_2^2$: by 
Theorem  \ref{propgen}  we just have to rule out the cases $m=\mu$ and $m\in (0,\mu)$.

If $m=\mu$, then $u_n$ tends to zero in $L^\infty_{loc}(\G)$ and, since $u_n \in V$, in $L^\infty(\G)$. As the $L^2$ norm of $u_n$ is fixed, $u_n \to 0$ in $L^p(\G)$, implying $\II = \lim_n E(u_n,\G) \ge 0$, a contradiction.
If $m \in (0,\mu)$, 
\[
\II >\min \left\{  \frac{\mu}m \ee_\R(m),\,
E(w,\G)\right\},
\]
 where $w\in H^1_\mu(\G)$ is the renormalized limit
\begin{equation}
\label{defw}
w(x)=\sqrt{\frac \mu{\mu-m}}\,\, u(x).
\end{equation}
But this is impossible, since
\[
\II >  \frac{\mu}m \ee_\R(m) >   \ee_\R(\mu)
\]
violates the assumptions
and, as $w\in V$,
\[
\II > E(w,\G) \ge \II
\]
is false. Thus $m=0$ and $u_n \to u$ strongly in $H^1(\G)$, with $E(u,\G) \le \II$. 
\end{proof}

\section{Proof of Theorem \ref{thm1}}
\label{sec:lambda}

In this Section we prove the first of the main results of the paper, Theorem \ref{thm1}. 
We begin with some preliminary well-known facts and some lemmas that are of interest in their own as they display  general features that may be shared by other classes of graphs (see Remark \ref{rem:extensions} at the end of this Section).
The first lemma establishes a very useful monotonicity property of the functions $\Lp$ and $\Lm$ defined in \eqref{Lpm}.

\begin{lemma}
\label{monot}
Let $\mu_1,\mu_2$ be such that $\mu_1 < \mu_2$, $M(\mu_1)\neq\emptyset$, $M(\mu_2)\neq\emptyset$ and $\Lambda^+(\mu_1)$, $\Lambda^-(\mu_2)$ are attained. Then
$\Lp(\mu_1) < \Lm(\mu_2)$.
\end{lemma}

\begin{proof} Let $\mu_1,\mu_2$ be as above and $u_i \in M(\mu_i)$, $i=1,2$, satisfy $\LL(u_1) = \Lp(\mu_1)$, $\LL(u_2) = \Lm(\mu_2)$. Define
$v_1,v_2\in H_1^1(\G)$  by setting  $u_1=\sqrt{\mu_1}v_1$, $u_2=\sqrt{\mu_2}v_2$, respectively. Then, since $p>2$ and $\mu_1<\mu_2$,
\[
\LL(u_1)=\LL(\sqrt{\mu_1}v_1)=\mu_1^{\f p2-1}\|v_1\|_p^p-\|v_1'\|_2^2 < 
\mu_2^{\f p2-1}\|v_1\|_p^p-\|v_1'\|_2^2 = \LL(\sqrt{\mu_2}v_1).
\]
Thus, if we manage to prove that $\LL(\sqrt{\mu_2}v_1)\leq\LL(\sqrt{\mu_2}v_2)=\LL(u_2)$,  we conclude.

To see this we start from the  two obvious inequalities 
\[
E(\sqrt{\mu_2}v_2, \G) \le E(\sqrt{\mu_2}v_1, \G), \qquad E(\sqrt{\mu_1}v_1, \G) \le E(\sqrt{\mu_1}v_2, \G) 
\]
that read, after rearranging terms,
\begin{equation}
\label{uno}
\|v_2'\|_2^2-\|v_1'\|_2^2\leq\f 2p \mu_2^{\f p2 -1}\left(\|v_2\|_p^p-\|v_1\|_p^p\right)
\end{equation}
and
\[
\|v_1'\|_2^2-\|v_2'\|_2^2\leq\f 2p \mu_1^{\f p2 -1}\left(\|v_1\|_p^p-\|v_2\|_p^p\right).
\]
Coupling them, we see that 
\[
\mu_1^{\f p2 -1}\left(\|v_2\|_p^p-\|v_1\|_p^p\right) \le \mu_2^{\f p2 -1}\left(\|v_2\|_p^p-\|v_1\|_p^p\right),
\]
from which we deduce $\|v_2\|_p^p-\|v_1\|_p^p \ge 0$. Finally, by \eqref{uno},
\begin{align*}
\LL(\sqrt{\mu_2}v_1) &- \LL(\sqrt{\mu_2}v_2) =  \mu_2^{\f p2 -1}\left(\|v_1\|_p^p-\|v_2\|_p^p\right) - 
\|v_1'\|_2^2 + \|v_2'\|_2^2 \\
& \le  \mu_2^{\f p2 -1}\left(\|v_1\|_p^p-\|v_2\|_p^p\right) + \frac2p \mu_2^{\f p2 -1}\left(\|v_2\|_p^p-\|v_1\|_p^p\right) \\
&= \mu_2^{\frac{p}2 -1} \left(1-\frac2p\right)\left(\|v_1\|_p^p-\|v_2\|_p^p\right) \le 0.
\end{align*}
\end{proof}

\begin{remark}
\label{monlem}
Note that, under the assumptions of Theorem \ref{thm1}, $M(\mu)\neq\emptyset$ for every $\mu\in J$, so that Lemma \ref{monot} guarantees that $\Lm,\,\Lp$ are strictly increasing functions on $J$.
\end{remark}

We are now ready to prove the first of our main results.

\begin{proof}[Proof of Theorem \ref{thm1}] 
For every $\mu \in J$,  Proposition \ref{newcomp} guarantees that $M(\mu)$ is compact, so that $\Lambda^-(\mu)$, $\Lambda^+(\mu)$ are attained, and Remark \ref{monlem} shows that they are strictly increasing functions on $J$. Let $\overline\mu$ be an internal point of $J$. For every $\delta>0$ small we can write, by Lemma \ref{monot} and the monotonicity of $\Lm$,
\[
\Lp(\overline\mu - 2\delta) < \Lm(\overline\mu - \delta) < \Lm(\overline\mu) < \Lm(\overline\mu+\delta) \le \Lp(\overline\mu +\delta).
\]
Now if $\Lp$ is continuous at $\overline\mu$, letting $\delta \to 0$ we obtain simultaneously
\[
\lim_{\mu \to \overline \mu} \Lm(\mu) = \Lm(\overline\mu)\quad\text{ and } \quad \Lm(\overline\mu) = \Lp(\overline\mu),
\]
thus proving that also $\Lm$ is continuous at $\overline\mu$. A symmetric computation starting with a point where $\Lm$ is continuous shows that $\Lm$ is continuous at $\mu$ if and only if $\Lp$ is continuous at $\mu$ and moreover that at every $\mu$ where these functions are continuous, $\Lm(\mu) = \Lp(\mu)$, namely that $\LL$ is constant on $M(\mu)$.

Let $Z \subset J$ be the set of points where $\Lp$ (and hence $\Lm$) is not continuous. Since $\Lp$ is increasing, the set $Z$ is at most countable. Setting $I = J \setminus (Z \cup \partial J)$, the preceding argument shows that  $\LL$ is constant on $M(\mu)$, for every $\mu \in I$. Hence we can define $\overline\lambda : I \to \R$ by
\[
\overline\lambda(\mu) = \LL(u), \quad u \in M(\mu),
\]
which is strictly increasing by Remark \ref{monlem}.

To conclude, we only have to prove  \eqref{derlevel}. For $\mu \in J$, let $u\in M(\mu)$ such that $\LL(u) = \Lp(u)$. Let $\eps >0$ and note that 
\begin{align}
\label{exp1}
E \left(\sqrt{\frac{\mu +\eps}{\mu}}u,\G\right) &= \frac12 \left(1+\frac\eps\mu\right)\int_\G |u'|^2\dx-
\frac1p \left(1+\frac\eps\mu \right)^{p/2}\int_\G |u|^p\dx \nonumber \\
& =E(u,\G) +\frac\eps{2\mu}\left(\int_\G |u'|^2\dx -\int_\G |u|^p\dx\right) +o(\eps) \nonumber \\
& = E(u,\G) -\frac\eps 2 \LL(u) +o(\eps) = E(u,\G) -\frac\eps 2 \Lp(\mu)+o(\eps)
\end{align}
as $\eps \to 0$.
Therefore,
\[
\ee_\G(\mu + \eps) -\ee_\G(\mu) \le E \left(\sqrt{\frac{\mu +\eps}{\mu}}u,\G\right) - E(u,\G) =  -\frac\eps 2 \Lp(\mu)+o(\eps).
\]
Since $\ee_\G$ is concave (Remark \ref{condlevel}), it has left and right derivatives $(\ee_\G)'_-$ and  $(\ee_\G)'_+$  everywhere. Dividing by $\eps$ and letting it tend to zero, we obtain
\begin{equation}
\label{halfde}
(\ee_\G)'_+(\mu) \le -\frac12 \Lp(\mu).
\end{equation}
Next, for every $\eps > 0$, take $u_\eps \in M(\mu +\eps)$. Similarly to \eqref{exp1},
\[
E \left(\sqrt{\frac{\mu}{\mu+\eps}}u_\eps,\G\right) 
= E(u_\eps,\G) +\frac\eps 2 \LL(u_\eps) +o(\eps)
\]
and therefore 
\begin{equation}
\label{dx}
\ee_\G(\mu + \eps) -\ee_\G(\mu) \ge  E(u_\eps,\G) - E \left(\sqrt{\frac{\mu +\eps}{\mu}}u_\eps,\G\right)  =  -\frac\eps 2 \LL(u_\eps)+o(\eps).
\end{equation}
as $\eps \to 0$. 
Let $(u_n)$ be a subsequence of $(u_\eps)$ with $\mu_n := \|u_n\|_2^2 \to \mu$ as $n\to \infty$. By Proposition \ref{newcomp}, $(u_n)$ has a subsequence (not relabeled) such that $u_n$ converges strongly in $H^1(\G)$ to some $u\in M(\mu)$. By the continuity of $\LL$ and Lemma \ref{monot}, 
\[
\Lp(\mu) < \Lm(\mu_n) \le \LL(u_n) =  \LL(u) + o(1) \le \Lp(\mu) +o(1),
\]
namely $\LL(u_n) \to  \Lp(\mu)$ as $n\to \infty$. Since this happens for every subsequence of $(u_\eps)$, we deduce that $\LL(u_\eps) \to \Lp(\mu)$ as $\eps \to 0$ and hence, dividing \eqref{dx} by $\eps$ and letting it tend to $0$, we obtain
\[
(\ee_\G)'_+(\mu) \ge  -\frac12 \Lp(\mu),
\]
which coupled with \eqref{halfde}, establishes the first part of \eqref{derlevel}. For the second part one simply starts with $u \in M(\mu)$ such that $\LL(u) = \Lm(\mu)$ and then works with $\eps <0$, repeating the same arguments.
\end{proof}

\begin{remark}
\label{rem:extensions}
The whole discussion of this Section is rather general, and it can be  adapted to apply to other families of graphs (e.g. without half-lines). Specifically, the presence of half--lines affects the previous arguments only through the assumption $\elevel_\G(\mu)<\elevel_\R(\mu)$, which is used to rule out sequences of ground states vanishing weakly in $H^1(\G)$. Therefore, aiming at a generalization of our results to other classes of graphs, this condition should be replaced by requiring that $\elevel_\G(\mu)$ be smaller than the infimum of the energies attained by sequences converging weakly to $0$ in $H^1_\mu(\G)$. To give concrete examples, in the case of periodic graphs without half-lines, this infimum energy is $0$ (see e.g. \cite{adst}, \cite{dovetta}), whereas in the case of binary metric trees it is $\frac12\lambda_1\mu$, where $\lambda_1$ is the bottom of the spectrum of the Laplacian on the graph (see \cite{dst}).
\end{remark}

\section{Uniqueness: proof of Theorems \ref{thm2} and \ref{thm3}}
\label{uni}
In this Section we establish uniqueness of ground states of given mass on two specific classes of graphs, taking advantage of the characterization of Lagrange multipliers given in Theorem \ref{thm1}. We first discuss the case of graphs with one terminal edge and $N$ half-lines all emanating from the same vertex, proving the results reported in Theorem \ref{thm2}. Next we consider the tadpole graph and we provide the proof of Theorem \ref{thm3}. 

We start with the following simple but general result, that we will use also in the next Section.

\begin{lemma}
\label{same} 
Let $\G$ be any metric graph and let $u,v \in H^1(\G)$ be solutions of \eqref{equation} having the same mass $\mu$,  the same energy level and the same Lagrange multiplier $\lambda$. Then 
\begin{equation}
\label{alleq}
\int_\G |u'|^2 \dx = \int_\G |v'|^2 \dx \qquad\hbox{and}\qquad \int_\G |u|^p \dx = \int_\G |v|^p \dx\,.
\end{equation}
\end{lemma}

\begin{proof} Since the energy level is the same,
\[
\frac12\int_\G |u'|^2 \dx - \frac12\int_\G |v'|^2 \dx = \frac1p\int_\G |u|^p \dx - \frac1p\int_\G |v|^p \dx
\]
while, from $\mu \LL(u) = \lambda = \mu \LL(v)$,
\[
\int_\G |u'|^2 \dx - \int_\G |v'|^2 \dx = \int_\G |u|^p \dx -\int_\G |v|^p \dx.
\]
These two equalities immediately imply \eqref{alleq}.
\end{proof}	
	
\subsection{Graphs with $\boldsymbol{N}$ half--lines and a terminal edge}
\label{subsec: N l}

For every $N\geq2$ and $t>0$, let $\G_{N,t}$ be the graph of Theorem \ref{thm2}  (Figure \ref{fig-NlT}.a). We identify the terminal edge ${\ti}$ of length $t$ with the interval $[0,t]$ so that $x=0$ corresponds to the vertex $\vv$ and $x=t$ to the tip of the edge. Every half-line is as usual identified with (a copy of) the interval $[0,+\infty)$.

We begin by recalling the shape of ground states on $\G_{N,t}$, and in particular their monotonicity and symmetry properties.
		
\begin{lemma}
\label{qualit1}
Let  $u\in\Hmu(\G_{N,t})$ be a ground state of mass $\mu$ on $\G_{N,t}$. Then $u$ is strictly increasing on the terminal edge while on each half-line it is a copy of the same, strictly decreasing, function. In particular, $u$ attains its maximum on the tip of the terminal edge.

\end{lemma}
\begin{proof} The proof follows working as in the proof of Theorem 2.7 in \cite{ast}. See also Theorem 2.9 in \cite{ast2}.
\end{proof}

\begin{proof}[Proof of Theorem \ref{thm2}] Let $t>0$ be fixed. The existence statements and the proof of \eqref{livelli_mu} are exactly the content of Theorem 2.6 in \cite{ast} (case $N=2$) and Theorem 4.4 in \cite{ast2} (case $N\ge 3$).
	
Let us thus turn to the uniqueness properties. Let $\mu\geq\overline\mu$ be such that all ground states with mass $\mu$ have the same Lagrange multiplier $\lambda$. By Theorem \ref{thm1}, this is true for all but at most countably many $\mu$. Assume that $u$ and $v$ are ground states of mass $\mu$. Since $u$ solves \eqref{equation}, the mechanical energy is a conserved quantity on each edge and, in particular,
\begin{equation}
\label{econs}
\frac12 |u'(x)|^2 + \frac1p |u(x)|^p  - \frac{\lambda}2 |u(x)|^2 = \begin{cases} 0 & \text{on each halfline} \\ C_u &\text{on the terminal edge} \end{cases}
\end{equation}
for some $C_u\in \R$ and all $x\in\G_{N,t}$.
The same holds for $v$, with a possibly different constant $C_v$. However, integrating \eqref{econs} and using \eqref{alleq},
\begin{align*}
t C_u &= \frac12 \int_{\G_{N,t}}  |u'|^2 \dx +  \frac1p \int_{\G_{N,t}} |u|^p \dx - \frac{\lambda}2  \mu \\
& =  \frac12 \int_{\G_{N,t}}  |v'|^2 \dx +  \frac1p \int_{\G_{N,t}} |v|^p \dx -  \frac{\lambda}2\mu = t C_v,
\end{align*}	
showing that $C_u = C_v$. 

By the Kirchhoff condition, $u'$ and $v' $ vanish at the tip $x=t$ of the terminal edge, and by  Lemma \ref{qualit1}, they have there a (positive) maximum point. Thus, from \eqref{equation},
\begin{equation}
\label{max}
u(t)^{p-1} - \lambda u(t) \ge u''(t) +  u(t)^{p-1} - \lambda u(t) = 0,
\end{equation}
i.e., $u(t) \ge \lambda^\frac1{p-2}$, and the same for $v$. Since the function $f(s) = \frac1p s^p -\frac{\lambda}2 s^2$ is injective  for $s \ge \lambda^{\frac1{p-2}}$, from
\begin{equation}
\label{maxeq}
f(u(t)) = C_u = C_v = f(v(t)),
\end{equation}
we deduce that $u(t) = v(t)$. Since the derivatives of $u$ and $v$ at $t$ are both zero, we see that $u\equiv v$ on the terminal  edge. But then also $u(0) =v(0)= :a$ and $u'(0) = v'(0)= :b$. Therefore on each halfline $u$ and $v$ (that solve the same differential equation) start with the same value $a$ and the same derivative $-b/N$, showing that they coincide on the whole of $\G$.
\end{proof}
	
\subsection{The tadpole graph}
\label{subsec: tad}
	
We now consider the tadpole graph $\T$ in Figure \ref{fig-NlT}.b. We denote by $\hh$  the half--line, by $\es$ the self-loop of $\T$ and by $2s:=|\es|$ its length. We identify $\es$ with the interval $[-s,s]$ so that $x=-s$ and $x=s$ correspond to the unique vertex $\vv$ of $\T$, and $\hh$, as usual, with $[0,+\infty)$.

%It is well--known (\cite[Corollary 3.4]{ast2}) that, for every $\mu>0$, ground states at mass $\mu$ on $\T$ always exist. 

The following lemma provides a description of  ground states with respect to the structure of the graph.
	
\begin{lemma}
\label{lem_tad}
Let  $u\in\Hmu(\T)$ be a ground state of mass $\mu$ on $\T$. Then $u$ is strictly decreasing on $\hh$, and $u$ is symmetric with respect to the point $x=0$ on $\es$ and strictly decreasing on $[0,s]$. In particular, $u$ attains its maximum on $\es$, at $x=0$.

\end{lemma}
	
\begin{proof} Let $u\in\Hmu(\T)$ be a ground state of mass $\mu$. Since $u$ solves \eqref{equation}, the set $\{x\in \T\; :\; u(x) = \sigma \}$ has measure zero for every $\sigma >0$. It is then easy to see that there exists $\tau >0$ such that $A_\tau :=\{x \in \T \; :\; u(x) >\tau \}$ has measure $2s = |\es|$. Clearly $u(A_\tau)$ is connected and contains $\|u\|_\infty$. 
Now it is readily seen that almost every value in $u(A_\tau)$ is attained at least twice in $A_\tau$. Indeed this is immediate for almost all the values attained in $A_\tau\cap\es$. On the other hand, if there exists  some value $\sigma \in u(A_\tau)$ attained only in $\hh$, then $\max_{x \in \hh} u(x) \ge \sigma > \max_{x \in \es} u(x)$ and, in particular, $\sigma>u(\vv)$. Since on $\hh$ it holds $u(0)=u(\vv)$ and $u(x)\to 0$ as $x\to+\infty$, it follows that every value between $\max_{x \in \es} u(x)$ and  $\max_{x \in \hh} u(x)$ is attained at least twice in $\hh$.
Thus
\begin{equation}
\label{ge2}
\#\{x \in A_\tau \; :\; u(x) = \sigma\} \ge 2 \qquad \text{for almost every } \sigma \ge \tau.
\end{equation}
On the other hand, trivially,
\begin{equation}
\label{ge1}
\#\{x \in \T \setminus A_\tau \; :\; u(x) = \sigma\} \ge 1 \qquad \text{for almost every } \sigma \in(0, \tau).
\end{equation}
Denote by $v$ the restriction of $u$ to $A_\tau$ and by $w$ its restriction to $\G\setminus A_\tau$. Let $\widehat v:[-s,s] \to \R$ be the symmetric rearrangement of $v$ (see \cite{ast}). Since $\widehat v$ satisfies $\widehat v(-s) = \widehat v(s)$, it can be seen as an element of $H^1(\es)$.
Finally, let $w^*:[0,+\infty) \to \R$ be the decreasing rearrangement of $w$. Plainly, $w^*$ is an element of $H^1(\hh)$.  Defining $\varphi : \T \to \R$ by 
\[
\varphi(x) = \begin{cases} \widehat v(x) & \text{ if } x \in \es \\
w^*(x)  & \text{ if } x \in \hh, \end{cases}
\] 
we see that  $\varphi \in H_\mu^1(\T)$, as $\widehat v(\vv) = w^*(\vv) = \tau$, and
\begin{equation}
\label{prep}
E(\varphi, \T) = E(\widehat v,{\es}) + E(w^*,{\hh}) \le E(v, A_\tau) + E(w, \T\setminus A_\tau) = E(u,\T).
\end{equation}
We claim that \eqref{ge2} and \eqref{ge1} are in fact equalities for almost every $\sigma$. Indeed, if either of the two holds with the strict inequality on a set of positive measure, then by standard properties of rearrangements (see \cite[Proposition 3.1]{ast}) we obtain the strict inequality in \eqref{prep}, contradicting the minimality of $u$. 

Now suppose that $|A_\tau \cap \hh | = :\delta >0$; in this case, $| A_\tau \cap {\es} | = 2s-\delta$, and $|{\es} \setminus A_\tau | = \delta$. Since 
\begin{equation}
\label{ge3}
\#\{x \in {\es} \setminus A_\tau \; :\; u(x) = \sigma\} \ge 2  \qquad \text{for almost every } \sigma \in (\min_{\es} u, \tau)
\end{equation}
and ${\es} \setminus A_\tau \subset  \T \setminus A_\tau$, this violates the fact already proved that \eqref{ge1} is an equality.
Therefore $|A_\tau \cap \hh | = 0$, namely $\T \setminus A_t = \hh$. Since \eqref{ge1} is an equality, this shows 
that $u$ is strictly decreasing on $\es$.

The symmetry properties of $u$ on $\T$ follow easily by standard phase plane analysis.
\end{proof}

\begin{proof}[Proof of Theorem \ref{thm3}]
The existence statements are proved in Section 3 in \cite{ast2} ($p\in(2,6)$) and in Theorem 3.3 in \cite{ast3} (p=6). The rest of the proof is very similar to the one of Theorem \ref{thm2}. Let $\mu >0$ be such that all ground states with mass $\mu$ have the same Lagrange multiplier $\lambda$. By Theorem \ref{thm1}, this is true for all but at most countably many $\mu$. Assume that $u$ and $v$ are ground states of mass $\mu$. 

Using Lemma \ref{same}, we can prove that the energy constants $C_u$ and $C_v$ on $\es$ coincide. Since $x=0$ is a  maximum point for $u$ and $v$ on $	\es$, we deduce as in \eqref{max} and \eqref{maxeq} that $u(0) = v(0)$. Since $u'(0) = v'(0)$, we see that $u \equiv v$ on $\es$ and, in particular, $u(s) = v(s) = : a$ and $u'(s) = v'(s)  = : b <0$. Then $u$ and $v$ start on the halfline with the same value $a$ and, by symmetry of $u$ and $v$ about $\vv$ and the Kirchhoff condition, with the same derivative $2b$, showing that $u$ and $v$ coincide on all of $\T$.
\end{proof}

\section{Non--uniqueness: proof of Theorem \ref{thm4}}
\label{nonuni}

This Section is devoted to the proof of the non-uniqueness result. The proof is rather long and requires several intermediate steps, some of which are general in nature and do not use the specific form of the graph $\G$ mentioned in Theorem \ref{thm4}.

From now on, let $p\in(2,6)$ and $\mu>0$ be fixed. We begin by proving some preliminary lemmas. The problem they address can be described in the following way: we consider a graph $\G_\ell$ containing an edge $\ed_\ell$ of variable length $\ell$ and we set
\[
V_\ell = \{ u \in H^1_\mu(\G_\ell) \; : \; \|u\|_{L^\infty(\G_\ell)} = \|u\|_{L^\infty(\ed_\ell)} \}, \qquad \ee_\ell(\mu) = \inf_{u\in V_\ell}E(u,\G_\ell).
\]
We are interested in the behavior of $\ee_\ell(\mu)$ as $\ell \to 0$, or $\ell \to \infty$ when $\ed_\ell$ is either a terminal edge or a self-loop. The study of these {\em doubly constrained} problems is central in the proof of Theorem \ref{thm4}.

Let us start with the discussion of the regime $\ell \to \infty$.

\begin{lemma}
\label{above}
Let $\G_\ell$ be a graph containing a bounded edge $\ed_\ell$ of (variable) length $\ell$. There results
\begin{equation}
\label{supfromabove}
\limsup_{\ell \to \infty} \ee_\ell(\mu) \le \begin{cases} \ee_{\R^+}(\mu) & \text{ if $\ed_\ell$ is a terminal edge} \\
\ee_\R(\mu) & \text{otherwise.} \end{cases}
\end{equation}
The limit holds uniformly with respect to the lengths of the other edges of $\G_\ell$.
\end{lemma}

\begin{proof} It is essentially Lemma 3.2 in \cite{astbound}. We give a short proof for completeness. Assume first that $\ed_\ell$ is a terminal edge of $\G_\ell$. Let $\phi$ be the half soliton of mass $\mu$ on $\R^+$ (Remark \ref{remsol}). For every $\eps>0$,
by standard density arguments, there exists $\phi_\eps \in H_\mu^1(\R^+)$
with \emph{compact support}, such that
\[
E(\phi_\eps,\R^+) \le (1-\eps)E(\phi, \R^+) = (1-\eps)\ee_{\R^+}(\mu).
\]
Now, for every  $\ell$ large enough, the length of the support of $\phi_\eps$
becomes smaller than $\ell$: thus we can fit $\phi_\eps$ on $\ed_\ell$
so that its maximum is attained at the tip of $\ed_\ell$.
Setting $\phi_\eps$ equal to zero on $\G_\ell\setminus \ed_\ell$, we may regard $\phi_\eps$ as a function
in $V_\ell$. Therefore
\[
\ee_\ell(\mu) \le E(\phi_\eps,\G_\ell) = E(\phi_\eps,\R^+) \leq (1-\eps)\ee_{\R^+}(\mu),
\]
which proves, via the arbitrariness of $\eps$ and $\ell$, that $\limsup _{\ell \to \infty} \ee_\ell(\mu) \le \ee_{\R^+}(\ell)$.

If $\ed_\ell$ is not a terminal edge, we repeat the preceding discussion starting this time from the soliton of  mass $\mu $ centered at zero. The same construction as above allows one to construct a function $\phi_\eps \in V_\ell$ such that $E(\phi_\eps,\G_\ell) \le (1-\eps) \ee_\R(\mu)$ for every $\ell$ large, from which the required inequality follows. 

The uniformity of the limit is obvious, since the construction used in the proof does not depend  on the lengths of the other edges of $\G_\ell$.
\end{proof}

\begin{remark}
\label{equality} 
If $\G_\ell$  contains at least one half-line, it is well known (see e.g. Theorem 2.2. in \cite{ast}) that 
$\ee_{\G_\ell}(\mu) \ge \ee_{\R^+}(\mu)$. Therefore in this case the first inequality in \eqref{supfromabove} becomes
\[
\lim_{\ell \to \infty}  \ee_\ell(\mu) =  \ee_{\R^+}(\mu).
\]
\end{remark}

Next we establish an estimate from below when $\ed_\ell$ is a self-loop.

\begin{lemma}
\label{loopbelow}
Let $\G_\ell$ be a noncompact graph containing a self-loop $\ed_\ell$ of (variable) length $\ell$. There results
\[
\liminf _{\ell \to \infty} \ee_\ell(\mu) \ge  \ee_{\R}(\mu).
\]
The limit holds uniformly with respect to the lengths of the other edges of $\G_\ell$.
\end{lemma}

\begin{proof} For every $u\in V_\ell$, let $\delta_\ell = \delta_\ell(u) = \min_{\ed_\ell} u$. Since
\[
\mu \ge \int_{\ed_\ell} |u|^2\dx \ge \ell \delta_\ell^2,
\]
we see that $\delta_\ell \to 0$ as $\ell\to \infty$. Let $p=p_\ell \in \ed_\ell$ be a point where $u(p) = \delta_\ell$ and let $\G_\ell'$ be the graph obtained from $\G$ with the addition of an extra terminal edge $\ed$ of length $1$ with one endpoint attached to $\G_\ell$ at $p$. On $\G_\ell'$ define a function $v$ by
\[
v(x) = \begin{cases} u(x) & \text{ if } x \in \G_\ell \\
\delta_\ell(1-x) & \text{ if } x \in {\ed} \sim [0,1]. \end{cases}
\]
Clearly $\nu_\ell := \|v\|_{L^2(\G_\ell')} = \mu + O(\delta_\ell^2)$ as $\ell\to \infty$ and, likewise,
\begin{equation}
\label{vu}
E(v,\G_\ell') = E(u,\G_\ell) + O(\delta_\ell^2).
\end{equation}
Now 
\[
\# \{x \in \G_\ell\; : \; u(x) = t\} \ge 2 \quad\text{for every } t \in (\delta_\ell, \|u\|_\infty),
\]
since $\ed_\ell$ is a self-loop on which $u$ attains its maximum while, obviously,
\[
\# \{x \in \G_\ell\; : \; u(x) = t\} \ge 1 \quad\text{for every } t \in (0, \delta_\ell).
\]
The function $v$ however takes every value $t \in [0, \delta_\ell)$ also on the extra edge $\ed$, so that
\[
\# \{x \in \G_\ell'\; : \; v(x) = t\} \ge 2 \quad\text{for almost every } t \in (0,  \|v\|_\infty).
\]
By \eqref{Ncontr},
\[
E(v,\G_\ell') \ge \ee_{\R}(\nu_\ell) =  \ee_{\R}(\mu) + O(\delta_\ell^2).
\]
and, in view of \eqref{vu},
\[
E(u,\G_\ell) = E(v,\G_\ell') + O(\delta_\ell^2) \ge  \ee_{\R}(\mu) + O(\delta_\ell^2).
\]
Letting $\ell \to \infty$, we conclude. Note that the argument does not depend on the lengths of the other edges, so that the limit is uniform. 
\end{proof}

To proceed with the limits as $\ell \to 0$ we need the following estimate. We recall from \cite{ast2} that the {\em compact core} of a noncompact graph $\G$ is the  graph obtained from $\G$ by removing every half-line.

\begin{lemma}
\label{minneg}
Let $\G$ be a graph made up of a compact core $\K$ of length $L$ and $N \ge 1 $ half-lines. For every $\mu>0$, there exists $\delta = \delta(L,N,\mu)>0$ such that if
\[
V = \{u\in H^1_\mu(\G)\; : \;  \|u\|_{L^\infty(\G)} = \|u\|_{L^\infty(\K)} \},
\]
then
\[
\inf_{u\in V} E(u,\G) \le -\delta.
\]
\end{lemma}

\begin{proof} Let $\phi_m$ be a half soliton of mass $m$, and recall from Remark \ref{remsol} that $\max \phi_m = \phi_m(0) = cm^\alpha$, with $\alpha = 2/(6-p)$ and $c$ depending only on $p$. Let
\[
v(x) = \begin{cases} \phi_m(x) & \text{ on each half-line of $\G$ } \\
cm^\alpha  & \text{ on  $\K$. } \end{cases}
\]
Clearly $v\in V$. Imposing that the mass of $v$ is $\mu$ amounts to
\begin{equation}
\label{massau}
c^2m^{2\alpha}L + Nm = \mu,
\end{equation}
while the energy of $v$ (by \eqref{energiamezzosol}) is
\[
E(v,\G) = -\frac1p c^p m^{\alpha p} L - N\theta_p 2^{2\beta}m^{2\beta+1} \le 
- N\theta_p 2^{2\beta}m^{2\beta+1}.
\]
Now from \eqref{massau}, since $m\le \mu$,
\[
\mu = c^2m^{2\alpha}L + Nm \le c^2 \mu^{2\alpha-1} mL +Nm, 
\]
namely
\[
m \ge \frac{\mu}{c^2\mu^{2\alpha -1}L +N}.
\]
Inserting this in the estimate for $E$ yields
\[
\inf_{u\in V} E(u,\G)  \le E(v,\G) \le -N\theta_p 2^{2\beta}  \frac{\mu^{2\beta+1}}{(c^2\mu^{2\alpha -1}L +N)^{2\beta+1}} =:-\delta,
\]
and the claim is proved. Note also that the estimate is sharp if $L=0$, namely if the graph consists of $N$ half-lines emanating from a single vertex, for in this case the estimate reduces to
\[
\inf_{u\in V}  E(u,\G) \le -\theta_p \left(\frac2N\right)^{2\beta} \mu^{2\beta+1},
\]
which cannot be improved, due to \eqref{Ncontr}.
\end{proof}

\begin{lemma}
\label{baffoazero} 
Let $\G_\ell$ be a graph containing  either a terminal edge or a self-loop $\ed_\ell$ of (variable) length $\ell$ attached at some fixed vertex $\vv$ and let
\[
V_0= \{ u \in H^1_\mu(\G_0) \; : \; \|u\|_{L^\infty(\G_0)} = u(\vv)\}\,,\qquad \ee_0(\mu) = \inf_{u\in V_0}E(u,\G_0)\,.
\]
There results
\begin{equation}
\label{limite0}
\lim_{\ell \to 0} \ee_\ell(\mu) = \ee_0(\mu). \end{equation}
The limit holds uniformly with respect to the lengths of the other edges of $\G_\ell$.
\end{lemma} 

\begin{proof} We identify the edge $\ed_\ell$ with the interval $[0,\ell]$ in such a way that $\vv$ corresponds to $x=0$ when $\ed_\ell$ is a terminal edge, and to $x= 0$ and $x=\ell$ when $\ed_\ell$ is a self-loop.

For every $\ell>0$, let $u_\ell \in V_\ell$ satisfy $E(u_\ell,\G_\ell) \le \ee_\ell(\mu) + \sqrt{\ell}$. By Lemma \ref{bdedseq} and the uniform boundedness of $\elevel_\ell(\mu)$ as $\ell \to 0$, we can assume that $\|u_\ell\|_{H^1(\G_\ell)}$ is uniformly bounded with respect to $\ell$.  Since $\|u_\ell\|_\infty$ is attained on $\ed_\ell$, 
\[
\|u_\ell\|_\infty \le u_\ell(0) +C\sqrt{\ell} = u_\ell(\vv) + C\sqrt{\ell},
\]
for some $C>0$ independent of $\ell$. Then, by Lemma  \ref{minneg},
\begin{align*}
-\delta +\sqrt{\ell} &\ge \ee_\ell(\mu) +\sqrt{\ell} \ge E(u_\ell,\G_\ell)  \ge -\frac1p\int_{\G_\ell} | u_\ell |^p \dx \\
& \ge -\frac1p\|u_\ell\|_\infty^{p-2}\mu \ge -\frac1p \left(u_\ell(\vv) +C\sqrt{\ell}\right)^{p-2} \mu
\end{align*}
from which we see that $\liminf_{\ell \to 0} u_\ell(\vv) >0$. 
 
Let $A_\ell = \{x \in \G_0\; : \; u_\ell(x) \ge u_\ell(\vv)\}$. Since
\[
\mu \ge \int_{A_\ell} |u_\ell |^2\dx \ge |u_\ell(\vv)|^2 |A_\ell |
\]
and $u_\ell(\vv)$ is bounded away from $0$, we deduce that $|A_\ell | \le C$ for every $\ell$,
for some $C$ independent of $\ell$. 

Notice that as $\ell \to 0$, by the uniform boundedness of $u_\ell$,
\begin{equation}
\label{uqbded}
\int_{\ed_\ell} |u_\ell |^q\dx \le \ell \|u_\ell \|_\infty^q = o(1)
\end{equation}
and 
\begin{equation}
\label{diffbded}
0 \le \int_{A_\ell} |u_\ell |^q -|u_\ell(\vv)|^q \dx \le  \Big(\left( u_\ell(\vv)  + C\sqrt{\ell} \right)^q - u_\ell(\vv)^q\Big) |A_\ell | = o(1)
\end{equation}
for every $q \ge 1$.

Now define a function $v_\ell : \G_0 \to \R$ as
\[
v_\ell(x) = \min\{u_\ell(x), u_\ell(\vv)\} = \begin{cases} u_\ell(\vv) & \text{ if } x \in A_\ell  \\
u_\ell(x) & \text{ if } x \in \G_0\setminus A_\ell \end{cases}
\]
Clearly, $v_\ell \in V_0$ and 
\[
\int_{\G_0} | v_\ell'|^2 \dx \le \int_{\G_\ell} | u_\ell'|^2 \dx.
\]
Furthermore, since as $\ell \to 0$
\begin{align*}
\int_\Gz |\vell| ^q \dx &= \int_{\Gz\setminus \All} |\vell |^q \dx + \int_\All |\vell |^q \dx \\ 
& = \int_{\Gz\setminus \All} |\uell |^q \dx + \int_\All |\uell |^q \dx + \int_\All |\vell |^q \dx-\int_\All |\uell |^q \dx \\
& = \int_\Gell |\uell |^q \dx - \int_{e_\ell}  |\uell |^q \dx + \int_\All |\uell(\vv) |^q \dx-\int_\All |\uell |^q \dx \\
& = \int_\Gell |\uell |^q \dx + o(1)
\end{align*}
for every $q\ge 1$ by \eqref{uqbded} and \eqref{diffbded}, we see that
\[
\int_\Gz |\vell |^2 \dx = \int_\Gell |\uell |^2 \dx + o(1) = \mu + o(1)
\]
and
\[
\int_\Gz |\vell |^p  \dx = \int_\Gell |\uell |^p \dx + o(1).
\]
Therefore, by our choice of $\uell$,
\[
\ee_\ell(\mu) + \sqrt{\ell} \ge E(\uell,\Gell) \ge E(\vell,\Gz) + o(1) \ge \ee_0(\mu +o(1)) + o(1) = \ee_0(\mu) + o(1),
\]
showing that
\begin{equation}
\label{liminf}
\liminf_{\ell \to 0} \ee_\ell(\mu) \ge \ee_0(\mu).
\end{equation}
On the other hand, given any $\eps>0$ and $u\in V_0$ such that $E(u,\Gz) \le \ee_0(\mu) + \eps$, we can extend $u$ to a function $\uell$ on $\Gell$ by setting $\uell(x) = \uell(0) = u(\vv)$ on $\ed_\ell$.
Now $\uell \in V_\ell$ and  $\|\uell\|_2^2 = \|u\|_2^2 +o(1) = \mu + o(1)$ as $\ell \to 0$, so that
\[
\ee_\ell(\mu) + o(1) = \ee_\ell(\mu+o(1)) \le E(\uell,\Gell) \le E(u,\Gz) \le \ee_0(\mu) + \eps,
\]
whence
\[
\limsup_{\ell \to 0} \ee_\ell(\mu) \le \ee_0(\mu).
\]
This, together with \eqref{liminf}, proves \eqref{limite0}.
\end{proof}

We now turn our attention to the graph of Figure \ref{grafone}. It is made up of $N\ge 2$ half-lines, a terminal edge $\ti$ and a bounded edge $\er$ all emanating from a single vertex $\vv$. A self-loop $\es$ is attached at the tip $\ww$ of the bounded edge $\er$. We denote by $r,s,t$ the lengths of the edges $\er, \es, \ti$ respectively. The graph itself will be denoted by $\G$ for generic values of $r,s,t$, and by $\G(r,s,t)$ if some specific values of the lengths play a role.

Next we define
\[
\Vuno = \{u \in H_\mu^1(\G)\; : \; \|u\|_{L^\infty(\G)} =  \|u\|_{L^\infty(\ti)}\}
\]
and
\[
\Vdue = \{u \in H_\mu^1(\G)\; : \; \|u\|_{L^\infty(\G)} =  \|u\|_{L^\infty(\es)}\}
\]
and we consider the infimum of the energies on $\Vuno$ and $\Vdue$ as  functions of the lengths $r,s,t$. Precisely, we define
$\Funo, \Fdue : (0,+\infty)^3 \to \R$ as
\[
\Funo(r,s,t) = \inf_{u \in \Vuno} E(u,\G),\qquad \Fdue(r,s,t) = \inf_{u \in \Vdue} E(u,\G),
\]
that are easily shown to be continuous.

We begin by observing that if $\G_0 = \G(r,s,0)$ and 
\[
V_0= \{ u \in H^1_\mu(\G_0) \; : \; \|u\|_{L^\infty(\G_0)} = u(\vv)\}, 
\]
then every $v \in V_0$ satisfies
\[
\# \{x \in \G_0\; : \; v(x) = \tau\} \ge N \quad\text{for almost every } \tau \in (0,  \|v\|_\infty).
\]
Thus, by  \eqref{Ncontr},
\[
\inf_{v\in V_0} E(v,\G_0) \ge \left(\frac2N\right)^{2\beta} \ee_\R(\mu)
\]
and the conclusion of Lemma \ref{baffoazero} reads
\begin{equation}
\label{prec}
\lim_{t \to 0} \Funo(r,s,t) \ge \left(\frac2N\right)^{2\beta} \ee_\R(\mu).
\end{equation}
Similarly, if this time  $\G_0 = \G(r,0,t)$ and 
\[
W_0= \{ u \in H^1_\mu(\G_0) \; : \; \|u\|_{L^\infty(\G_0)} = u(\ww)\}, 
\]
then by Lemma \ref{baffoazero},
\begin{equation}
\label{limit2}
\lim_{s\to 0} \Fdue(r,s,t) = \inf_{u\in W_0} E(u,\G_0).
\end{equation}
The next lemma is crucial for the proof of Theorem \ref{thm4}.

\begin{lemma}
\label{mira} For every $r$ large enough and every $\eps>0$ small enough, there exist 
$\overline s = \overline s(r,\eps)$ and $\overline t = \overline t(r,\eps)$ such that
\[
\Funo(r,  \overline s,  \overline t) = \Fdue(r,  \overline s,  \overline t) = \ee_\R(\mu) -\eps.
\]
%The values $\Funo$ and $\Fdue$ are achieved by.... (metterlo qui o dopo?) 
\end{lemma}

\begin{proof} Let $2\eps_0 = \ee_\R(\mu) - \ee_{\R^+}(\mu)$ and let $r_0$ be so large that 
\begin{equation}
\label{inf0}
\inf_{u\in W_0} E(u,\G_0) \le \ee_{\R^+}(\mu) + \eps_0.
\end{equation}
The choice of $r_0$ is possible since, as in Lemma \ref{above}, when $r\to \infty$ the above infimum tends to $\ee_{\R^+}(\mu)$.
Fix $r \ge r_0$ and $\eps \in(0,\eps_0)$.
By  \eqref{prec},
\begin{equation}
\label{tzero}
\lim_{t\to 0} \Funo(r,s,t) \ge   \left(\frac2N\right)^{2\beta} \ee_\R(\mu) > \ee_\R(\mu) -\eps,
\end{equation}
while by Lemma \ref{above} and Remark \ref{equality},
\begin{equation}
\label{tinfty}
\lim_{t\to \infty} \Funo(r,s,t) = \ee_{\R^+} (\mu)< \ee_\R(\mu) -\eps.
\end{equation}
Similarly, by our  choice of $r_0$, \eqref{limit2} and \eqref{inf0} show that
\begin{equation}
\label{szero}
\lim_{s\to 0} \Fdue(r,s,t) \le  \ee_{\R^+}(\mu) + \eps_0 < \ee_\R(\mu) -\eps,
\end{equation}
while by Lemma \ref{loopbelow},
\begin{equation}
\label{sinfty}
\liminf_{s\to \infty} \Fdue(r,s,t) \ge  \ee_\R(\mu).
\end{equation}
Now for every fixed $r\in (r_0,+\infty)$ and $\eps \in (0,\eps_0)$, define $F : (0,+\infty)^2 \to \R^2$ by 
\[
F(s,t) =(F_1(s,t),F_2(s,t)) =  (\Funo(r,s,t) -\ee_\R(\mu)+\eps,\Fdue(r,s,t) -\ee_\R(\mu)+\eps).
\]
The limit relations \eqref{tzero}--\eqref{sinfty} show that for every $r > r_0$ and every $\eps < \eps_0$, there exist $s_1,s_2$ and $t_1,t_2$ such that on $[s_1,s_2]\times[t_1,t_2]$ the map $F$ satisfies
\[
F_1 > 0 \text{ on } [s_1,s_2] \times \{t_1\}, \qquad   F_1 < 0 \text{ on } [s_1,s_2] \times \{t_2\},
\]
while
\[
F_2 < 0 \text{ on }  \{s_1\}\times [t_1,t_2], \qquad   F_2 > 0 \text{ on }  \{s_2\}\times [t_1,t_2].
\]
Hence, by the Brouwer--Miranda Theorem (\cite{Miranda}), there exist $(\overline s,\overline t)\in [s_1,s_2] \times [t_1,t_2]$ (of course depending on $r$ and $\eps$) such that
\[
F(\overline s,\overline t) = (0,0).
\]
Explicitly, this means that
\[
\Funo(r, \overline s, \overline t) = \Fdue(r, \overline s, \overline t) = \ee_\R(\mu) -\eps,
\]
as we wanted to prove.
\end{proof}

The next result shows that the values $\Funo(r, \overline s, \overline t)$ and $\Fdue(r, \overline s, \overline t)$ are achieved by functions that are in fact ground states of $E$ on $\G( r, \overline s, \overline t)$. To simplify notation we write $\G, \Vuno, \Funo$ for $\G( r, \overline s, \overline t), \Vuno( r, \overline s, \overline t), \Funo( r, \overline s, \overline t)$ and so on.

\begin{theorem}
\label{fattained} For every $r$ large enough and every $\eps>0$ small enough, let $\overline s$ and $\overline t$ be the lengths provided by the preceding lemma. Then the values $\Funo$ and $\Fdue$ are achieved. The functions $u$ and $v$ (depending on $r$ and $\eps$) that achieve $\Funo$ and $\Fdue$ are ground states of $E(\,\cdot\,,\G)$ over $H_\mu^1(\,\G)$. 
\end{theorem}

\begin{proof} The first part follows immediately from Theorem \ref{compact} since $\Funo = \Fdue  \le \ee_\R(\mu)$. We now prove the second part, namely that $u$ and $v$ are ground states on $H_\mu^1(\G)$.
First note that since $E(u,\G) = \Funo = \ee_\R(\mu)-\eps$, then $\ee_\G(\mu)< \ee_\R(\mu)$, and therefore $\ee_\G(\mu)$ is achieved by some function $w$ (Theorem 3.3. in \cite{ast2}). By Lemma \ref{noB2}, the maximum of $w$ is attained either on $\ti$ or on $\es$. Suppose it is attained on $\ti$ (the other case works in the same way). Thus, $w\in \Vuno$ and then
\[
\ee_\G(\mu) \le \Funo \le E(w,\G) = \ee_\G(\mu),
\]
showing that $\Funo = \ee_\G(\mu)$. Therefore $E(u,\G) = \Funo = \ee_\G(\mu)$, and $u$ is a ground state. Since $E(v,\G) = \Fdue = \Funo$, also $v$ is a ground state.
\end{proof}

We now analyze the behavior of the minimizers found in the preceding theorem when $r\to \infty$ and $\eps \to 0$. To this aim
we take two sequences $r_n \to \infty$ and $\eps_n\to 0$ and we  let $\overline{s}_n$ and $\overline{t}_n$ be the values provided by Lemma \ref{mira}. We set $\G_n = \G(r_n, \overline{s}_n, \overline{t}_n)$ and we call  $u_n$ and $v_n$  the functions in $\Vuno(r_n, \overline{s}_n, \overline{t}_n)$ and $\Vdue(r_n, \overline{s}_n, \overline{t}_n)$ that satisfy $E(u_n,\G_n) = E(v_n,\G_n) = \ee_\R(\mu)- \eps_n$ provided by Theorem \ref{fattained}.

We begin by studying the asymptotic behavior as $n\to \infty$ of $u_n$ and $v_n$. We recall that $\G_{N,t}$ denotes the graph made up of $N$ half-lines and a terminal edge of length $t$, all emanating from the same vertex $\vv$ (Figure \ref{fig-NlT}.a).

\begin{lemma} 
\label{moving}
Let $(r_n)$ and $(\eps_n)$ be sequences such that $r_n \to \infty$ and $\eps_n \to 0$ as $n\to \infty$. Then, up to subsequences,
\begin{itemize}
\item[i)] $\overline{s}_n \to \infty$;
\item[ii)] $\overline{t}_n \to t \in \R$, where $t$ is such that 
\begin{equation}
\label{slimit}
\ee_{\G_{N+1,t}}(\mu) = \ee_\R(\mu).
\end{equation}
\end{itemize}
\end{lemma}

\begin{proof} Assume that $(\overline{s}_n)$ is bounded and let $\T_n$ be the tadpole graph with a self-loop of length $\overline{s}_n$. From \cite{ast2}, we know that there is $\delta>0$ so that
\[
\ee_{\T_n}(\mu) \le \ee_\R(\mu) -2\delta,
\]
for every $n$.
By the usual density arguments, there exists $w_n \in H_\mu^1(\T_n)$
with compact support, such that
\[
E(w_n,\T_n)\le \ee_{\T_n}(\mu) + \delta.
\]

Now, for every  $r_n$ large enough, the length of the support of $w_n$
becomes smaller than $r_n$: thus, we can view $w_n$ as an element of $\G_n$
by setting it equal to zero on $\G_n\setminus {\rm spt}(w_n)$. Therefore
\[
\ee_{\G_n}(\mu) \le E(w_n,\G_n) = E(w_n,\T_n) \le \ee_{\T_n}(\mu) +\delta \le  \ee_\R(\mu) -\delta,
\]
for every $n$, which is impossible since $\ee_{\G_n}(\mu) = \ee_\R(\mu) -\eps_n \to \ee_\R(\mu)$ as $n\to \infty$ by assumption. Point $i)$ is proved.

To prove $ii)$ first note that if (along some subsequence) $\overline{t}_n \to \infty$, then, by Remark \ref{equality},
$\ee_{\G_n}(\mu) \to \ee_{\R^+}(\mu) < \ee_\R(\mu)$, contradicting again the assumptions. Thus $\overline{t}_n$ must be bounded and we can assume that it converges to some $t$ as $n\to \infty$. We are left to show that \eqref{slimit} holds. The inequality $\ee_{\G_{N+1,t}}(\mu) \le \ee_\R(\mu)$ is always true by Theorem \ref{thm2}.
Assume that the strict inequality holds. Since by continuity 
\[
\ee_{\G_{N+1,\overline{t}_n}}(\mu) = \ee_{\G_{N+1,t}}(\mu) + o(1)
\]
as $n\to \infty$, we also have $\ee_{\G_{N+1,\overline{t}_n}}(\mu) <  \ee_\R(\mu)$ for every $n$ large, and therefore this level is achieved by some $v_n \in H_\mu^1(\G_{N+1,\overline{t}_n})$  (Theorem \ref{thm2}). By Lemma \ref{qualit1}, the maximum of $v_n$ is attained at the tip of the terminal  edge ${\tt t}_n$ of $\G_{N+1,\overline{t}_n}$. Once again by density arguments, it is easy to see that there exist functions $w_n \in H_\mu^1(\G_{N+1,\overline{t}_n})$ with compact support such that
\[
E(w_n,\G_{N+1,\overline{t}_n}) = E(v_n, \G_{N+1,\overline{t}_n}) +o(1)
\]
as $n\to \infty$ (it is enough to take, for instance, $w_n = \frac{\sqrt \mu}{\|(v_n -\delta_n)^+\|_2}(v_n -\delta_n)^+ $ with $\delta_n \to 0$). Since $v_n$ attains its maximum on the tip of ${\tt t}_n$, so does $w_n$ and, once extended to $0$, the function $w_n$ can be seen as an element of $\Vuno(r_n, \overline{s}_n,\overline{t}_n)$. Then, as $n\to \infty$,
\begin{align*}
\ee_\R (\mu) -\eps_n &= \ee_{\G_n}(\mu) \le E(w_n, \G_n) = E(w_n, \G_{N+1,\overline{t}_n}) = E(v_n, \G_{N+1,\overline{t}_n}) + o(1) \\
& = \ee_{\G_{N+1,\overline{t}_n}}(\mu) +o(1) = \ee_{\G_{N+1,t}}(\mu) +o(1),
\end{align*}
contradicting the strict inequality  $\ee_{\G_{N+1,t}}(\mu) < \ee_\R(\mu)$.
\end{proof}

The next two lemmas establish the asymptotic behavior of the Lagrange multipliers $\LL(u_n)$ and $\LL(v_n)$.

\begin{lemma}
\label{lambdau}  Let $u_n$ be the ground state on $\G_n$ found in Theorem \ref{fattained}. Then
\[
\LL(u_n) \to \LL(w),
\]
where $w$ is a ground state of mass $\mu$ on $\G_{N+1,t}$.

\end{lemma}

\begin{proof} The functions $u_n$ satisfy $E(u_n,\G_n) = \ee_\R(\mu) -\eps_n$ and, as minimizers of $E$ over 
$\Vuno(r_n, \overline{s}_n, \overline{t}_n)$, attain their $L^\infty$ norm on $\ti_n$. Since $ \overline{s}_n\to \infty$, clearly $\delta_n := \min_{\es_n} u_n \to 0$ as $n\to \infty$.

We consider the graph $\G_{N+1,t}$ obtained from $\G_n$ by replacing the terminal edge $\ti_n$ with a terminal edge $\ti$ of length $t$ (identified with the interval $[0,t]$) and by replacing the self-loop $\es_n$ with a half-line (identified with the interval $[r_n,+\infty)$).

For every $n$, let $\psi_n \in H^1(\R^+)$ be the function constructed via Lemma \ref{stiro}. 
We define a sequence $\overline{w}_n \in H^1(\G_{N+1,t})$ by
\[
\overline{w}_n(x) = \begin{cases} u_n\left(\frac{\overline{t}_n}{t}x\right) & \text{ on the bounded edge $\ti$ of length $t$} \\
u_n(x) & \text{ on the half-lines and on the edge $\er_n$ of length $r_n$} \\
\psi_n(x-r_n) & \text{ on the half-line $[r_n, +\infty)$}. \end{cases}
\]
Since $\overline{t}_n \to t$ and $\delta_n \to 0$, by Lemma \ref{stiro} it is readily seen that $\nu_n:= \|\overline{w}_n\|_2^2 = \mu + o(1)$ as $n\to \infty$. Finally, set 
\[
w_n =\sqrt{\frac{\mu}{\nu_n}}\, \overline{w}_n \in H_\mu^1(\G_{N+1,t}).
\]
By the properties of $u_n$, $\overline{t}_n$, $\delta_n$, and Lemma \ref{stiro}, one easily checks that as $n\to \infty$,
\begin{equation}
\label{simili}
\int_{\G_{N+1,t}} |w_n'|^2\dx = \int_{\G_n} |u_n'|^2\dx+ o(1)
\end{equation}
and
\begin{equation}
\label{similip}
\int_{\G_{N+1,t}} |w_n|^p\dx = \int_{\G_n} |u_n|^p\dx+ o(1).
\end{equation}
From this and \eqref{slimit} it follows that 
\[
E(w_n,\G_{N+1,t}) = E(u_n,\G_n) +o(1) = \ee_\R(\mu) -\eps_n +o(1) = \ee_{\G_{N+1,t}} +o(1),
\]
namely that $w_n$ is a minimizing sequence for $E$ on $H_\mu^1(\G_{N+1,t})$. By construction, $w_n$ attains its $L^\infty$ norm on the bounded edge. This shows that minimizing $E$ on $H_\mu^1(\G_{N+1,t})$ is the same as minimizing it among functions that attain their maximum on the bounded edge. For this reason, by Theorem \ref{compact}, $w_n$ converges strongly (up to subsequences) in $H^1(\G_{N+1,t})$ to a ground state $w$. Thus, by \eqref{simili} and \eqref{similip},
\[
\mu\LL(u_n) = \int_{\G_n} |u_n|^p\dx - \int_{\G_n} |u_n'|^2\dx  \to 
\int_{\G_{N+1,t}} |w|^p\dx - \int_{\G_{N+1,t}} |w'|^2\dx  = \mu \LL(w),
\]
when $n\to \infty$, as we wanted to prove.
\end{proof}

\begin{lemma}
\label{lambdav}
Let $\phi_\mu$ be the soliton of mass $\mu$ on $\R$ centered at $0$, so that $E(\phi_\mu,\R) = \ee_\R(\mu)$, and let $v_n$ be the ground state on $\G_n$ found in Theorem \ref{fattained}. Then
\[
\LL(v_n) \to \LL(\phi_\mu)
\]
as $n\to \infty$.
\end{lemma}

\begin{proof} As in the preceding lemma,  $\delta_n := \min_{\es_n} v_n \to 0$ as $n\to \infty$.  Denote by $p_n \in \es_n$ a point where $u(p_n) = \delta_n$. We attach a terminal edge of unitary length to $\G_n$ at $p_n$ and we call $\G_n'$ the new graph. Following step by step the procedure used in the proof of Lemma \ref{loopbelow} we can construct a function $w_n \in H^1(\G_n')$ such that, as $n\to \infty$,
\begin{itemize}
\item[$i)$] $\nu_n := \|w_n\|_{L^2(\G_n')}^2 = \mu + o(1)$
\item[$ii)$] $\|w_n'\|_{L^2(\G_n')} = \|v_n'\|_{L^2(\G_n)} +o(1), \qquad  \|w_n\|_{L^p(\G_n')} = \|v_n\|_{L^p(\G_n)}+o(1)$
\item[$iii)$] $\# \{x \in \G_n'\; : \; w_n(x) = \tau \} \ge 2 \quad\text{for almost every } \tau \in (0,  \|w_n\|_\infty)$.
\end{itemize}
Now let $\widehat w_n$ be the symmetric rearrangement of $w_n$ on $\R$. By standard properties (see e.g. \cite{ast}) and the preceding relations,
\[
E(\widehat w_n,\R)  \le E(w_n,\G_n') = E(v_n,\G_n) + o(1) = \ee_\R(\mu) + o(1).
\]
On the other hand, $E(\widehat w_n,\R) \ge \ee_\R(\nu_n) = \ee_\R(\mu+o(1)) = \ee_\R(\mu) +o(1)$. Therefore
$E(\widehat w_n,\R) \to \ee_\R(\mu)$ as $n\to \infty$ from which one easily deduces that
\[
\widehat w_n \to \phi_\mu \quad\text{in } H^1(\R)
\]
as $n\to \infty$. Finally we observe that 
\begin{align*}
\frac12 \int_{\G_n} |v_n'|^2\dx &=  \ee_\R(\mu) + \frac1p \int_{\G_n} |v_n|^p\dx +o(1) =
\ee_\R(\mu) + \frac1p \int_{\G_n'} |w_n|^p\dx +o(1) \\
&= \ee_\R(\mu) + \frac1p \int_\R |\widehat w_n|^p\dx +o(1) = 
\ee_\R(\mu) + \frac1p \int_\R |\phi_\mu|^p\dx +o(1) \\
& = \frac12 \int_\R |\phi_\mu'|^2\dx +o(1)
\end{align*}
and from this it follows that as $n \to \infty$,
\[
\mu\LL(v_n) = \int_{\G_n} |v_n|^p\dx - \int_{\G_n} |v_n'|^2\dx  \to 
 \int_\R |\phi_\mu|^p\dx - \int_\R |\phi_\mu'|^2\dx  = \mu \LL(\phi_\mu),
\]
as we wanted to prove.
\end{proof}

\begin{proof}[End of the proof of Theorem \ref{thm4}] For every $\mu >0$ we can choose $r$ so large and $\eps$ so small that,
by Theorem \ref{fattained}, the graph $\G:=\G(r,\overline{s}, \overline{t})$ admits two ground states $u, v \in H_\mu^1(\G)$. 
Of course $u$ and $v$ satisfy \eqref{equation} with Lagrange multipliers $\LL(u)$ and $\LL(v)$ respectively.

Let $w$ be the ground state in $H_\mu^1(\G_{N+1,t})$ found in Lemma \ref{lambdau}, and recall that $E(w,\G_{N+1,t}) = \ee_{\G_{N+1,t}}(\mu) = \ee_\R(\mu)$ by construction.

By Lemma \ref{lambdau} and \ref{lambdav},  taking (if necessary) $r$ even larger and $\eps$ even smaller,
we can make sure that $\LL(u)$ is as close as we wish to $\LL(w)$, and that $\LL(v)$ is similarly close to $\LL(\phi_\mu)$.

Therefore, if we manage to show that $\LL(w) \ne \LL(\phi_\mu)$, then we obtain at the same time that $u \ne v$ and that $\LL(u) \ne \LL(v)$, concluding the proof of Theorem \ref{thm4}.

To this purpose, we assume that $\LL(w)= \LL(\phi_\mu) =: \lambda$ and we seek a contradiction. Using the fact that  $w$ and $\phi_\mu$ have the same mass, the same energy and the same Lagrange multiplier, one shows as in Lemma \ref{same}  that 
\[
 \int_{\G_{N+1,t}} |w'|^2 \dx = \int_\R |\phi_\mu'|^2 \dx \qquad\hbox{and}\qquad
 \int_{\G_{N+1,t}} |w|^p \dx = \int_\R |\phi_\mu|^p \dx .
\]	
Let $C$ be the mechanical energy conserved by $w$ on the bounded edge of $\G_{N+1,t}$. Integrating the energy conservation equality we thus see that
\begin{align*}
\frac12\lambda\mu &= \frac12\int_\R |\phi_\mu'|^2 \dx +\frac1p \int_\R |\phi_\mu|^p \dx \\ 
&= \frac12\int_{\G_{N+1,t}} |w'|^2 \dx + \frac1p\int_{\G_{N+1,t}} |w|^p \dx = \frac12\lambda\mu + C\ell,
\end{align*}
namely that $C=0$. This means that on the bounded edge $w$ coincides with the soliton $\phi_\mu$.  Plainly, also on each half-line $w$ coincides with $\phi_\mu$. Then, by the Kirchhoff condition, the  vertex of ${\G_{N+1,t}}$ is an absolute maximum point of $w$, and this violates the fact already proved that $w$ achieves its absolute maximum at the tip of the bounded edge.
\end{proof}

\appendix
\section{Appendix}

We prove here two technical results that are used in the preceding Section. We start by describing the location of maximum points of ground states on the graph $\G = \G(r,s,t)$ of Figure \ref{grafone}.

\begin{lemma} 
\label{noB2} Let $u$ be a ground state of mass $\mu$ on $\G$ such that $E(u,\G) < \ee_\R(\mu)$. Then $u$ attains its maximum in $\es \cup \ti$.
\end{lemma}

\begin{proof} The function $u$ attains its maximum in $\er \cup \es \cup \ti$ by Proposition 2.5 in \cite{ast2}. Assume that $u$ attains it in $\er$ only, and call $p$ the maximum point.
Set $\alpha = u(\vv)$ and identify the part of $\er$ between $p$ and $\ww$ with the interval $[p,r]$. Two alternatives may occur.
\medskip

\noindent $i)$ If $\displaystyle\min_{[p,r] \cup \es} u \le\alpha$, then
\[
\# \{x \in \G\; : \; u(x) = \tau \} \ge 2 \quad\text{for every } \tau \in (\alpha, \|u\|_\infty),
\]
because each of these values $\tau$ is attained at least once in $\er$ between $\vv$ and $p$ and at least once  in $[p,r] \cup \es$. Since there are $N\ge 2$ half-lines starting from $\vv$ we conclude that 
\[
\# \{x \in \G_r\; : \; u(x) = \tau \} \ge 2 \quad\text{for every } \tau \in (0, \|u\|_\infty),
\]
and then, by  \cite[Lemma 2.1]{astbound}), $E(u,\G) \ge \ee_\R(\mu)$, contrary to the assumptions.
\medskip

\noindent $ii)$ If $\displaystyle\min_{[p,r] \cup \es} u >\alpha$, let $q \in [0,p)$ be the last point where $u(q) = \alpha$. Since $u$ solves \eqref{equation}, its level sets have measure $0$ and hence there exists $\tau$ such that 
\[
A_\tau = \{ x \in [q,r] \cup {\es} \, : \; u(x) \ge \tau\}\
\]
has measure $s$. Note that $\tau > \alpha$ because $u(x) > \alpha$ in $(q,r]\cup \es$, that has measure greater than $s$, and that 
\begin{equation}
\label{posmeas}
|A_\tau \cap [q,r]| >0, 
\end{equation}
since otherwise it would be $A_\tau = \es$, and then $u$ would attain its maximum in $\es$, contradicting the assumption. This also shows that $\min_{\es} u <\tau$. Then
\[
\# \{x \in A_\tau \; : \; u(x) = \sigma \} \ge 2 \quad\text{for every } \sigma \in (\tau, \|u\|_\infty).
\]
We now consider the symmetric rearrangement $\widehat u$ of $u_{|_{A_\tau}}$, that can be seen as a function on $\es$ and  satisfies $\min_{\es} \widehat u = \tau$. Finally we consider the decreasing rearrangement $u^*$ of $u$ restricted to $[q,r]\cup {\es} \setminus A_\tau$. This is a function defined in $[0, r-q]$ and  satisfies $u^*(0) = \tau$ and $u^*(r-q) = \alpha$. Note that due to \eqref{posmeas},  $[q,r]\cup {\es} \setminus A_\tau$ intersects $\es$ in a set of positive measure and therefore every value attained by $u$ in this intersection is attained at least twice. Thus, by the usual properties of rearrangements, $\|(u^*)'\|_{L^2(0,r-q)} < \|u'\|_{L^2([q,r]\cup {\es} \setminus A_\tau)}$. Defining $v:\G \to \R$ as
\[
v(x) = \begin{cases} \widehat u(x) & \text{ if } x \in {\es} \\
u^*(r-x) & \text{ if } x \in [q,r] \\
u(x) & \text{ elsewhere in }  \G \end{cases}
\]
we obtain that $v \in H_\mu^1(\G)$ and $E(v,\G) < E(u,\G)$, contradicting the minimality of $u$.
\end{proof}		
		
The next result is used in the proof of Lemma \ref{lambdau}.

\begin{lemma}
\label{stiro} Let $\es$ be a self-loop of length $s$ and let $u$ be a nonnegative function in $H^1(\es)$. 
Set $\delta = \min_{\es } u$ and let $p$ be a fixed point on $\es$. Then there exists a function $\psi\in H^1(\R^+)$ such that
\begin{itemize}
\item[i)] $\psi(0) = u(p)$;
\item[ii)] $\displaystyle \int_0^{+\infty} |\psi|^2\dx = \int_{\es} |u|^2\dx +  \frac12\delta^2$;
\item[iii)] $\displaystyle  \int_0^{+\infty} |\psi'|^2\dx \le \int_{\es} |u'|^2\dx  +  \frac12 \delta^2$;
\item[iv)] $\displaystyle\int_0^{+\infty}|\psi|^p\dx =\int_{\es}  |u|^p\dx  + \frac1p \delta^p$.
\end{itemize}
\end{lemma}
		
\begin{proof} Let $\tau = u(p)$, consider the level set $A_\tau =\{x \in {\es} \; : \; u(x) \ge \tau\}$ and let $\ell = |A_\tau| \ge 0$.  If $\ell >0$, then define $\widehat u:[0,\ell] \to \R$ to be the symmetric rearrangement of $u$ restricted to $A_\tau$. Thus,
\begin{equation}
\label{equi}
\int_0^\ell |\widehat u|^q\dx = \int_{A_\tau} |u|^q\dx \quad \forall q\in [1,+\infty)
\end{equation}
and, since $\# \{u^{-1}(\sigma)\} \ge 2$ for almost every $\sigma \in (\tau,\|u\|_\infty)$ as  $\es$ is a self-loop,  
\begin{equation}
\label{polya}
\int_0^\ell |(\widehat u)'|^2\dx \le \int_{A_\tau} |u'|^2\dx.
\end{equation}
Next, define $ u^*:[\ell,r] \to \R$ to be the decreasing rearrangement of $u$ restricted to ${\es}\setminus A_\tau$, that satisfies the two above properties as well. We note that $\widehat u(0) = u(p) = \widehat u(\ell)$ and that $u^*(\ell) = \widehat u(\ell)$, $u^*(s) = \delta$.
Finally define $\psi:\R^+ \to \R$ by
\[
\psi(x) = \begin{cases} \widehat u(x) & \text{ if } x \in [0,\ell) \\
u^*(x) & \text{ if } x \in [\ell, s] \\
\delta e^{s-x} & \text{ if } x >s \end{cases}
\]
(if $\ell = 0$ then simply neglect $\widehat u$ in the definition of $\psi$). Plainly, $\psi\in H^1(\R^+)$ and $\psi(0) = u(p)$. Furthermore, by \eqref{equi} and a direct computation,
\[
\int_0 ^{+\infty}| \psi |^q\dx =  \int_0^\ell | \widehat u |^q\dx + \int_\ell^s | u^* |^q\dx + \frac{\delta^q}q = \int_{\es} | u |^q\dx + \frac{\delta^q}q, \quad \forall q\in [1,+\infty),
\]
while, by \eqref{polya},
\[
\int_0^{+\infty} | \psi' |^2\dx =  \int_0^\ell | (\widehat u)' |^2\dx + \int_\ell^s | (u^*)' |^2\dx + \frac{\delta^2}2 \le \int_{\es} | u' |^2\dx + \frac{\delta^2}2,
\]
from which $ii)$, $iii)$ and $iv)$ follow.
\end{proof}


\begin{thebibliography}{99}
	
\bibitem{acfn1}
Adami R., Cacciapuoti C., Finco D., Noja D., 
{\em Constrained energy minimization and orbital stability for the NLS equation on a star graph}, 
Ann. Inst. Poincaré (C) An. Non. Lin. {\bf 31} (6) (2014), 1289–-1310.
	
\bibitem{acfn2} 
Adami R., Cacciapuoti C., Finco D., Noja D., {\em Stable standing waves for
a NLS on star graphs as local minimizers of the constrained energy},
J. Differential Equations {\bf 260} (10) (2016), 7397--7415.
	
\bibitem{ad}
Adami R., Dovetta S., 
{\em One-dimensional versions of three-dimensional system: ground states for the NLS on the spatial grid},
Rendiconti di Matematica e delle sue applicazioni, \textbf{39} (7) (2018), 181--194.

\bibitem{adr}
Adami R., Dovetta S., Ruighi A., 
{\em Quantum graphs and dimensional crossover: the honeycomb}, 
Comm. Appl. Ind. Math., \textbf{10} (1), (2019), 109--122.

\bibitem{adst}  Adami R., Dovetta S., Serra E., Tilli P., 
\textit{Dimensional crossover with a continuum of critical exponents for NLS on doubly periodic metric graphs}, 
Anal. PDE {\bf 12} (2019), no. 6, 1597--1612. 
	
\bibitem{ast}
Adami R., Serra E., Tilli P., 
\textit{NLS ground states on graphs},
Calc. Var. and PDEs \textbf{54} (1), 743–761 (2015).
	
\bibitem{ast2}
Adami R., Serra E., Tilli P., 
\textit{Threshold phenomena and existence results for NLS ground states on graphs},
J. Funct. An. \textbf{271} (1), 201-223 (2016).

\bibitem{ast3}
R. Adami, E. Serra, P. Tilli, 
Negative energy ground states for the $L^2$-critical NLSE on metric graphs,
\emph{Comm. Math. Phys.} {\bf 352} (2017), no. 1, 387--406.

\bibitem{ASTparma}
Adami R., Serra E., and   Tilli P., \textit{Nonlinear dynamics on branched structures and networks}, Riv. Mat. Univ. Parma \textbf{8}, no. 1 (2017),  109--159.

\bibitem{astbound}
Adami R., Serra E., Tilli P., 
\textit{Multiple positive bound states for the subcritical NLS equation on metric graphs}, 
Calc. Var. PDEs {\bf 58} (2019) no. 5. 16pp.

\bibitem{BKKM}
Berkolaiko G., Kennedy J.B., Kurasov P., Mugnolo D.,
\emph{Surgery principles for the spectral analysis of quantum graphs}
Trans. Amer. Math. Soc. {\bf372} (2019), 5153--5197.

\bibitem{BK}
Berkolaiko G., Kuchment P.,
\emph{Introduction to quantum graphs},
Mathematical Surveys and Monographs 186, American Mathematical Society, Providence, RI, 2013.

\bibitem{BLS}
Berkolaiko G., Latushkin Y., Sukhtaiev S.,
\emph{Limits of quantum graph operators with shrinking edges}, Adv. Math. {\bf352} (2019), 632--669.

\bibitem{BMP}
Berkolaiko G., Marzuola J.L., Pelinovsky D., \emph{Edge-localized states on quantum graphs in the limit of large mass}. ArXiv:1910.03449 (2019).

\bibitem{BL} Brezis H., Lieb E. H., 
\textit{A relation between pointwise convergence of functions and convergence of functional}, 
Proc. Amer. Math. Soc. {\bf 88} (3) (1983) 486--490.

\bibitem{cds}
Cacciapuoti C., Dovetta S., Serra E.,
{\em Variational and stability properties of constant solutions to the NLS equation on compact metric graphs}, 
Milan Journal of Mathematics, \textbf{86} (2) (2018), 305--327.

\bibitem{cfn}
Cacciapuoti C., Finco D., Noja D., 
\textit{Ground state and orbital stability for the NLS equation on a general starlike graph with potentials}, 
Nonlinearity \textbf{30} (2017), 3271--3303.

\bibitem{CFN2}
Cacciapuoti F., Finco D., Noja D., 
\emph{Topology--induced bifurcations for the nonlinear Schr\"odinger equation on the tadpole graph},
Phys. Rev. E \textbf{91} (1), 013206 (2015).
	
\bibitem{cazenave}
Cazenave T., Semilinear Schr\"odinger Equations, Courant Lecture Notes 10. American
Mathematical Society, Providence, RI, 2003.

\bibitem{dancer}
Dancer E.N.,
{\em The effect of the domain shape on the number of positive solutions of certain nonlinear equations},
J. Differential Equations {\bf 74} (1988), 120--156.

\bibitem{dovjde} Dovetta S.,
{\em Existence of infinitely many stationary solutions of the $L^2$--subcritical and critical NLSE on compact metric graphs}, 
J. Differential Equations \textbf{264} (2018), no. 7, 4806--4821.
	
\bibitem{dovetta}  Dovetta S.,
\textit{Mass-constrained ground states of the stationary NLSE on periodic metric graphs},
NoDEA Nonlinear Differential Equations Appl. {\bf 26} (2019), no. 5,  30 pp.

\bibitem{dst} Dovetta S., Serra E., Tilli P.,
\textit{NLS ground states on metric trees: existence results and open questions},
J. London Math. Soc., to appear. ArXiv:1905.00655.

\bibitem{dt-p}
Dovetta S., Tentarelli L., 
{\em Ground states of the $L^2-$critical NLS equation with localized nonlinearity on a tadpole graph}, Operator Theory: Advances and Applications, to appear. ArXiv:1803.09246.

\bibitem{dt} Dovetta S., Tentarelli L., 
{\em $L^2$--critical NLS on noncompact metric graphs with localized nonlinearity: topological and metric features}, 
Calc. Var. PDE Vol. 58, Is. 3 (2019) 58:108, https://doi.org/10.1007/s00526-019-1565-5.

\bibitem{EFK}
Ekholm T., Frank R.L., Kovarik H., \emph{Eigenvalue estimates for Schr\"odinger operators on metric trees}, Adv. Math. {\bf226} (2011), no. 6, 5165--5197.

\bibitem{EP}
Exner P., Post O., {\em Approximation of quantum graph vertex couplings by scaled Schr\"odinger operators
on thin branched manifolds}, J. Phys. A: Math. Theo. {\bf 42} (2009), 415305, 22pp.

\bibitem{gnn}
Gidas B., Ni W.M., Nirenberg L.,
{\em Symmetry and related properties via the maximum principle},
Comm. Math. Phys. {\bf 68} (1979), 209--243.

\bibitem{g}
Goodman R. H.,
\textit{NLS Bifurcations on the bowtie combinatorial graph and the dumbbell metric graph},
Discr. Cont. Dyn. Systems - A, (2019), {\bf39} (4) : 2203--2232.

\bibitem{KMPX}
Kairzhan A., Marangell R., Pelinovsky D., Xiao K.L., 
\emph{Standing waves on a flower graph}. ArXiv:2003.09397 (2020).

\bibitem{KKMM}
Kennedy J.B., Kurasov P., Malenov\'a G., Mugnolo D., \emph{On the Spectral Gap of a Quantum Graph},
Ann. Henri Poincar\'e {\bf17}, (2016), 2439--2473.	

\bibitem{KS}
Kostrykin V., Schrader R., \emph{Kirchhoff's rule for quantum wires}, J. Phys.
A: Math. Gen. {\bf32} (1999), 595--630.

\bibitem{mp}
Marzuola J. L.,  Pelinovsky D., 
\textit{Ground state on the dumbbell graph},
Appl. Math. Res. Express \textbf{2016}, no. 1 (2016), 98--145.
	
\bibitem{Miranda}
Miranda C., {\em Un'osservazione su un teorema di Brouwer}, Boll. Unione Mat. Ital. {\bf 3} (1940), 527. 

\bibitem{N}
Noja D., \textit{Nonlinear Schr\"odinger equation on graphs: recent results and open problems}, Phil. Trans. R. Soc. A \textbf{372} (2007) (2014), 20130002.

\bibitem{NP}
Noja D.,  Pelinovsky D.,
\emph{Standing waves of the quintic NLS equation on the tadpole graph}. ArXiv:2001.00881 (2020).

\bibitem{nps}
Noja D.,  Pelinovsky D., Shaikhova G., 
\emph{Bifurcations and stability of standing waves in the nonlinear Schr\"odinger equation on the tadpole graph}, 
Nonlinearity \textbf{28} (2015),  2343--2378.

\bibitem{pa}
Pankov A., 
\textit{Nonlinear Schr\"odinger equations on periodic metric graphs}, 
Discrete Contin. Dyn. Syst. {\bf 38} (2) (2018), 697--714.

\bibitem{ps}
Pelinovsky D., Schneider G., 
\textit{Bifurcations of Standing Localized Waves on Periodic Graphs}, 
Ann. H. Poincar\'e {\bf 18} (4) (2017), 1185--1211.

\bibitem{PSV}
Pierotti D., Soave N., Verzini G.,
\emph{Local minimizers in absence of ground states for the critical NLS energy on metric graphs}. ArXiv:1909.11533 (2019).

\bibitem{st}
Serra E., Tentarelli L., 
Bound states of the NLS equation on metric graphs with localized nonlinearities,
\emph{J. Differential Equations} {\bf 260} (2016), no. 7, 5627--5644.

\bibitem{st2}
Serra E., Tentarelli L.,
On the lack of bound states for certain NLS equations on metric graphs,
\emph{Nonlinear Anal.} {\bf 145} (2016), 68--82.

\bibitem{t}
Tentarelli L., NLS ground states on metric graphs with localized nonlinearities,
\emph{J. Math. Anal. Appl.} {\bf 433} (2016), no. 1, 291--304.

\end{thebibliography}
\end{document}